\documentclass[12pt]{article}
\usepackage{dsfont}
\usepackage{multicol,amssymb,amsmath}
\usepackage{mathrsfs}
\usepackage{amsthm}
\usepackage{graphicx}
\usepackage{xcolor}
\usepackage{enumerate}
\usepackage[margin=15truemm]{geometry}
\usepackage{tikz}
\usepackage{tikz-cd}
\usepackage{ifthen}
\usetikzlibrary{positioning}
\usetikzlibrary{svg.path}
\usetikzlibrary{
	math,
	bbox,
	knots,
	hobby,
	decorations.pathreplacing,
	shapes.geometric,
	calc
}

\numberwithin{equation}{section}
\newtheorem{thm}{Theorem}[section]
\newtheorem*{thm*}{Theorem}
\newtheorem{lem}[thm]{Lemma}
\newtheorem{prop}[thm]{Proposition}

\newtheorem{cor}[thm]{Corollary}
\theoremstyle{definition}\newtheorem{defn}[thm]{Definition}
\newtheorem{eg}[thm]{Example}
\newtheorem*{rmk}{Remark}
\newcommand{\HH}{{\mathcal H}}
\newcommand{\HHH}{{\mathbb H}}
\newcommand{\LL}{{\mathcal L}}
\newcommand{\KK}{{\mathcal K}}
\newcommand{\FF}{{\mathcal F}}

\newcommand{\RR}{{\mathbb R}}

\newcommand{\SSS}{{\mathbb S}}

\newcommand{\CC}{{\mathbb C}}
\newcommand{\NN}{{\mathbb N}}

\newcommand{\TT}{{\mathcal T}}
\newcommand{\ZZ}{{\mathbb Z}}

\newcommand{\frakA}{\mathfrak{A}}
\newcommand{\ps}{\text{PS}}

\newcommand{\Cr}{\text{wcr}}

\newcommand\crossing[5]{
	\begin{tikzpicture}
		\foreach \x in {#1,...,#2}
		\coordinate (\x) at (\x,0) node [below] at (\x) {\x};
		\foreach \x/\y [count=\n] in {#3}
		{   \draw (\x)--(\x,\n);
			\draw (\y)--(\y, \n);
			\draw (\x,\n)--(\y, \n);}
		\foreach \x in {#4}
		\draw (\x)--(\x, #5);
	\end{tikzpicture}
}

\begin{document}
	\title{A Conjugate System for Twisted Araki-Woods von Neumann Algebras of finite dimensional spaces}
	\author{Zhiyuan Yang\footnote{ 
					 Texas A\&M University, zhiyuanyang@tamu.edu}}
	\maketitle
	
\begin{abstract}
	We compute the conjugate system of twisted Araki-Woods von Neumann algebras $ \mathcal{L}_T(H) $ introduced in \cite{dL22} for a compatible braided crossing symmetric twist $T$ on a finite dimensional Hilbert space $ \HH $ with norm $ \|T\| <1$. This implies that those algebras have finite non-microstates free Fisher information and therefore are always factors of type $\text{III}_\lambda$ ($0<\lambda\leq 1$) or $ \text{II}_1 $. Moreover, using the (nontracial) free monotone transport \cite{nelson2015free}, we show that $ \mathcal{L}_T(H) $ is isomorphic to the free Araki-Woods algebra $ \mathcal{L}_0(H) $ when $ \|T\|=q $ is small enough.
\end{abstract}
\section{Introduction}
The study of operators with the deformed commutation relations
$$ a_ia^*_j = \sum_{r,s}t^{i\,r}_{j\,s}a^*_{r}a_s+\delta_{i\,j}\text{id}, \quad 1\leq i,j\leq d $$
started with \cite{BS94} by Bo{\.{z}}ejko and Speicher and also independently \cite{JSW93} by Jorgensen, Schmitt and Werner. When the operator $T\in B(\RR^d \otimes \RR^d)$ with the matrix coefficients $ t^{i\,r}_{j\,s} = \langle e_i\otimes e_r,T(e_j\otimes e_s) \rangle $ (called the twist operator) is self-adjoint, satisfying the braided relation (Yang-Baxter relation) $$ (\text{id}\otimes T)(T\otimes \text{id})(\text{id}\otimes T)=(T\otimes \text{id})(\text{id}\otimes T)(T\otimes \text{id}),$$ and has norm $\|T\|\leq 1$,  \cite{BS94} shows that the deformed commutation relations can be realized by the creation and annihilation operators on the $ T $-twisted Fock space $ \FF_T(\CC^d) $. The von Neumann algebra $ \mathcal{L}_{T}(\RR^d )$ generated by the self-adjoint operators $ X_i:= a_i+a^*_i $ on $ \FF_T(\CC^d) $ is called the $T$-twisted (or Yang-Baxter deformed) Gaussian algebra, which could be seen as a deformation of the free Gaussian algebras (or free group factors) for which $ T=0 $. When the twist $ T $ is the $q$-scaled tensor flip $T=qF$ with $-1<q<1$, then $ \mathcal{L}_{qF}(\RR^d)$ is the $q$-Gaussian algebra.

In \cite{dL22}, da Silva and Lechner generalized the above construction to the quasi-free cases, namely, we can replace $\RR^d$ (or in general a real Hilbert space $H_\RR$) by a nontrivial standard subspace $ H\subset \HH $ of a complex Hilbert space $ \HH $. The corresponding von Neumann algebra $ \mathcal{L}_T(H) $ is then called the $T$-twisted Araki-Woods von Neumann algebra since it is a deformation of the free Araki-Woods algebra \cite{shlyakhtenko1997free} ($T=0$) which is an analogy of the Araki-Woods factors ($T = -F$).

It is shown in \cite{dL22} that, a natural assumption to put on the twist $T$ is \begin{enumerate}[a)]
	\item  $T$ is compatible with the modular operator $ \Delta_H $ in the sense that $ [T,\Delta_H^{it}\otimes \Delta_H^{it}] = 0 $ for all $t\in \RR$,
	\item $ T $ is braided (satisfies the Yang-Baxter equation),
	\item $T$ satisfies a KMS type condition called crossing symmetry (see Definition \ref{crossing symmetry 1}).
\end{enumerate}
In fact, the vacuum state is separating for $ \mathcal{L}_T(H) $ with compatible $T$ if and only if $ T $ is braided and crossing symmetric. Under those assumptions, the modular data of the vacuum vector can be shown to be exactly the second quantization of the original modular data of $H\subset \HH$. For the tracial case (i.e. when $\HH$ is the complexification of $H=H_\RR$), the crossing symmetric condition becomes the cyclic relation $ t^{i\,r}_{j\,s} = t^{r\,s}_{i\,j} $ which already appeared in \cite{BS94}. When assuming that $ T $ is braided, this condition is also equivalent to the vacuum state being tracial. 

While the factoriality (as well as solidity and fullness) of the $q$-Araki-Woods algebras ($T = qF$) is now known for any standard subspace $ H\subset \HH $ due to the recent paper \cite{KSW23} based on a sequence of previous works including but not limited to \cite{bikram2022factoriality}\cite{skalski2018remarks}\cite{bikram2017generator}\cite{Fum01}, for a general twist $ T $ the factoriality of $ \mathcal{L}_T(H)$ is still unknown except for the tracial cases with certain assumption. In the recent preprint \cite{kumar2023conjugate}, using a dual system argument similar to \cite{MS23}\cite{KSW23}, it is also shown that mixed-$q$-Araki-Woods algebras are always non-injective factors.

 For the tracial case, if $ \|T\|<1 $ and $ \text{dim}\,H_\RR $ is finite but sufficiently large, \cite{Kr06} shows that the $T$-twisted Gaussian algebras $ \mathcal{L}_{T}({H_\RR} )$ must be type $\text{II}_1$ factors. Under the assumption that $ T $ has an eigenvector of the form $\xi_0\otimes \xi_0$, \cite{bikram2021neumann} shows that $\mathcal{L}_{T}({H_\RR} )$ is a $ \text{II}_1 $ for $\HH$ of any dimension. For infinite dimensional $H_\RR$, \cite{Kr00} (also \cite{bikram2021neumann}) shows that if the coefficients of $ T $ satisfy certain 'finitely supported' properties, then $ \mathcal{L}_{T}({H_\RR} ) $ is again a $\text{II}_1$ factor. In \cite{krolak2005contractivity}, a version of ultracontractivity of the $T$-twisted Ornstein-Uhlenbeck semigroup is also proved.

The proof of the factoriality of $q$-Araki-Woods in \cite{KSW23} essentially breaks into four cases: 1) $ \Delta_H^{it} $ has no almost periodic part (i.e. no eigenvectors). 2) $\Delta_H$ has a fixed vector $\xi$ ($\Delta_H^{it}$ is non-ergodic); 3) $H$ has a proper ergodic almost periodic two dimensional subspace; 4) The almost periodic part of $H$ has dimension $2$. Case 1) and Case 2) are due to \cite{bikram2017generator}\cite{skalski2018remarks} by showing that the centralizer of $ \mathcal{L}_{qF}(H) $ is always contained in the almost periodic part and that $ s_q(\xi) = X_q(\xi) $ generates a strongly mixing MASA. Case 3) is shown by \cite{bikram2022factoriality}. And the last case is proven via the conjugate system by \cite{KSW23}, generalizing the result in \cite{MS23} to the nontracial cases (after reducing the case to the almost periodic part using second quantization.) In fact, \cite{MS23}'s method directly proves the factoriality for all finite dimensional $H$, not just the case of dimension $2$.  The main difficulty of generalizing the arguments in the first three cases to $ T $-twisted algebras is that for a general twist $ T $, given a subspace $\KK\subseteq \HH$, we may not have $ T(\KK\otimes \KK)\subseteq \KK\otimes \KK $, which makes it hard to estimate the terms in $ \mathcal{L}_T(\HH) $. In fact, there exists nontrivial twist $T$ (Example \ref{twist on matrix}) such that any $S_H$-invariant subspace $\KK\subseteq \HH$ satisfying $ T(\KK\otimes \KK)\subseteq \KK\otimes\KK $ must be either $\HH$ or $ \{0\} $. Due to this reason, it is also difficult to use second quantization arguments to find expected subalgebras of $\mathcal{L}_T(\HH)$. It turns out that, however, the conjugate variable method in \cite{KSW23}\cite{MS23} is the part that ‘depends less' on the properties of the tensor flip $F$. And we will show that this method can be generalized to the twisted Araki-Woods algebras using similar arguments.

In this paper, we are mainly interested in the conjugate system of $ \mathcal{L}_T(H) $ when $H$ is finite dimensional and $ \|T\|< 1 $ and we want generalize the results in \cite{KSW23} to the case of a general twist $T$. Our computation is inspired by a formula about the free difference quotient of Wick polynomials for $q$-Gaussian algebras in Proposition 5.1. \cite{MS23}:
$$ \partial_i (e_{j_n}\otimes\cdots \otimes e_{j_1}) = \sum_{\pi\in C(n+1)}(-1)^{|p(\pi)|}q^{\text{cross}(\pi)-|s_r(\pi)|}\delta_{p(\pi)}e_{s_l(\pi)}\otimes e_{s_r(\pi)}, $$
where $ C(n+1) $ is certain family of partitions and $ s_l(\pi) $ ($s_r(\pi)$) is the set of singletons on the left (right) of a particular index $ k $. Roughly speaking, this particular formula suggests that one need to consider a certain decomposition of a general partition with respect to a specified singleton $k$. See Subsection 3.5 for a precise statement about this decomposition and Proposition \ref{formula for partial i} for a generalization of this formula.

Our main result is the following.
\begin{thm*}
	For finite dimensional standard subspace $H\subseteq \HH$ and compatible braided crossing symmetric twist $T$ with $ \|T\|<1 $, if $e_1,\cdots,e_d$ is a basis of $ \HH $, then the conjugate system $ (\Xi_1,\cdots,\Xi_d) $ for $ (X_T(e_1),\cdots,X_T(e_d))\subseteq \mathcal{L}_T(H) $ exists and has the concrete form
	$$ \Xi_i = \sum_{n=0}^{\infty}(-1)^{n}P_{T,2n+1}^{-1}\sum_{ \pi\in B(2n+1) }(W_\pi^T)^*f_i \in \FF_T(\HH) ,\quad \forall 1\leq i\leq d, $$
	where $ (f_1,\cdots,f_d) $ is the dual basis of $ ( e_1,\cdots,e_d ) $ in $ \HH $, and $ B(2n+1)\subseteq \mathcal{P}_{1,2}(2n+1) $ is the set of incomplete matchings such that $ n+1 $ is a singleton and each $ 1\leq k\leq n $ must be paired with an element in $ \{n+2,\cdots,2n+1\} $.
\end{thm*}
 Here $ W^T_\pi $'s are certain $T$-twisted contraction operators related to Wick products. The convergence of the series is due to the same reason as in the $q$-Gaussian cases \cite{MS23}: for each $ \pi \in B(2n+1) $, the norm of the twisted contraction operators $W^T_\pi$ decays rapidly in the sense that $ \| W^T_\pi \|\lesssim \|T\|^{\frac{n(n+1)}{2}}= q^{\frac{n(n+1)}{2}} $. In particular, this implies the factoriality of $\mathcal{L}_T(H)$ for all finite dimensional $H$ with $ \|T\|=q<1 $ by \cite{Ne17}. Also, one can use the (nontracial) free monotone transport \cite{nelson2015free} to show that $ \mathcal{L}_T(H) $ is isomorphic to the free Araki-Woods algebra for $\|T\|$ small enough.

Our computation for the conjugate system essentially relies on the $T$-Wick formula (Therem \ref{Wick product}), and thus our approach is different from \cite{MS23} (they first derived a formula for a dual system by induction and then constructed the conjugate system from the dual system). Note that a version of $T$-Wick formula in terms of operators of the form $ a^*_{i_1}\cdots a^*_{i_s}a_{j_1}\cdots a_{j_t}$ already exists for a general twist $ T $ by \cite{Kr00}. While this version of $T$-Wick formula is a powerful tool for example when constructing the second quantization, (also frequently used to show that a $\Delta_H$-fixed vector $\xi$ generates MASA), it is unclear how one can use it to compute the conjugate system. Due to this reason, we need to first derive a new $T$-Wick formula in terms of polynomials in $ X_T(e_1),\cdots,X_T(e_d) $ before we could actually compute the conjugate system.

The idea of our main computation is as follows: To calculate the conjugate variable $ \Xi_i = \partial_i^{*}(\Omega\otimes \Omega) $ (assuming it exists), it suffices to compute the free difference quotient $ \partial_i $ on some nice basis and then take the adjoint. There are two ways to look at $ \partial_i $. We can either consider $ \partial_i:\CC\langle x_1,\cdots,x_d\rangle\to \CC\langle x_1,\cdots,x_d\rangle\otimes \CC\langle x_1,\cdots,x_d\rangle $ as a map defined on polynomials, or consider $\partial_i:\FF_T(\HH)\to \FF_T(\HH)\otimes \FF_T(\HH)$ as a densely defined unbounded operator. While it is easy to compute $ \partial_i $ on polynomials (simply by definition), in order to compute $ \partial_i^*(\Omega\otimes \Omega) $, one is forced to apply $ \partial_i $ to Fock basis vector $ \xi_1\otimes\cdots\otimes \xi_n \in \FF_T(\HH)$ which corresponds to the Wick polynomial $ \varPhi(\xi_1\otimes\cdots\otimes \xi_n)\in \mathcal{L}_T(H) $. Therefore, to get a formula for $ \partial_i( \xi_1\otimes\cdots\otimes \xi_n ) $, one needs to first apply $T$-Wick formula, then compute the free difference quotient, and finally apply the inverse of $T$-Wick formula as in the diagram below.

\[
\begin{tikzcd}
	\FF_T(\HH)\ni\hspace{-6,5em}&\xi_1\otimes\cdots\otimes \xi_n\arrow[dashed]{d}{\partial_i}\arrow{r}{\varPhi}[swap]{\text{Thm. \ref{Wick product}} }& \varPhi(\xi_1\otimes\cdots\otimes \xi_n) \arrow{d}{\partial_i}\arrow{dl}{\text{Prop. \ref{formula for partial i} }}&\hspace{-6.5em}\in\mathcal{L}_T(H)\\
	\FF_T(\HH)\otimes \FF_T(\HH)\ni\hspace{-2.5em}&\partial_i(\xi_1\otimes\cdots\otimes \xi_n) & \partial_i\varPhi(\xi_1\otimes\cdots\otimes \xi_n) \ar[l,"\varPhi^{-1}\otimes \varPhi^{-1}"]&\hspace{-2.5em}\in\mathcal{L}_T(H)\otimes \mathcal{L}_T(H)^{op}
\end{tikzcd}
\]

The plan of the paper is as follows: In Section $2$, we give a quick introduction to the $T$-twisted Araki-Woods algebras and recall the basic notations about conjugate system. Some examples of compatible braided crossing symmetric twists on finite dimensional spaces are given. Section $3$ is devoted to the formula of the $T$-Wick product. And in Section $4$ we compute the conjugate system and establish the main results of this paper. In Section $5$, we apply the nontracial free monotone transport \cite{nelson2015free} to $\mathcal{L}_T(H)$ with small $ \|T\|=q $. We also collected some results about the non-injectivity of $ \mathcal{L}_T(H) $ for infinite dimensional $ \HH $ in the appendix using arguments similar to the corresponding arguments in \cite{Fum01} for $q$-Araki-Woods algebras.

\section{Preliminary}
\subsection{Twisted Fock space}
We begin with a brief introduction to the Twisted Fock space following \cite{dL22}. Let $ \HH $ be a Hilbert space, $ T\in B(\HH\otimes \HH) $ be a bounded self-adjoint operator with $ \|T\|\leq 1 $. For each $n\geq 2$, we define $ T_k = 1^{k-1}\otimes T\otimes 1^{n-k-1}\in B(\HH^{\otimes n}) $. Define inductively \begin{align}
	R_{T,n} &:= 1+T_1+T_1T_2+\cdots+ T_1\cdots T_{n-1},\\
	P_{T,1}&:= R_{T,1}, \quad P_{T,n} := (1\otimes P_{T,n-1})R_{T,n}.
\end{align}
In particular, $ P_{T,n} = (1\otimes P_{T,n-1})R_{T,n}= (1+T_{n-1})(1+T_{n-2}T_{n-1})\cdots (1+T_1+\cdots+T_1\cdots T_{n-1}) $, and one can show that $ P_{T,n} $ is self-adjoint whenever $ T $ is self-adjoint.
\begin{defn}
	A twist $ T $ is a operator in $$ \TT_{\geq }:=\{T\in B(\HH \otimes \HH): \|T\|\leq 1, \quad P_{T,n}\geq 0 \quad \forall n\in \NN_+\}. $$
	
	A strict twist $ T $ is a operator in $$ \TT_{> }:=\{T\in B(\HH \otimes \HH): \|T\|\leq 1, \quad P_{T,n} \mbox{ strictly positive} \quad \forall n\in \NN_+\}. $$ And a braided twist $ T $ is a twist $ T\in \TT_{\geq } $ that satisfies the Yang-Baxter equation $$ T_1T_2T_1 = T_2T_1T_2. $$
\end{defn}

\begin{thm}[\cite{BS94}]
	Let $ T\in B(\HH\otimes \HH) $ be a self-adjoint operator. If $ \|T\|\leq 1 $ and $ T $ is braided, then $ T\in \TT_{\geq } $. If furthermore $ \|T\|<1 $, then $ T\in \TT_{> } $.
\end{thm}

If $ T $ is a twist, for each $n\geq 1$, consider the quotient space $ \HH_{T,n}^0 =  \HH^{\otimes n}/\ker P_{T,n} = \overline{  \mbox{Ran}P_{T,n} }$, let $ \HH_{T,n} $ be the completion of this quotient space with respect to the inner product $ \langle \cdot,P_{T,n}\cdot \rangle $. For $n=0$, we set $ \HH_{T,0}=\CC\Omega  $ where $ \Omega $ is a unit vector called the vacuum vector.

\begin{defn}
	Let $ T\in \TT_{\geq} $ be a twist, the $T$-twisted Fock space is the Hilbert space:
	$$ \FF_T(\HH) = \bigoplus_{n=0}^{\infty}\HH_{T,n},$$
	where the inner product $ \langle \cdot,\cdot\rangle_{T} $ is $ \langle f,g \rangle_{T}  = \sum_{n=0}^{\infty}\langle f_n,P_{T,n}g_n\rangle$ for $ f_n,g_n\in \HH_{T,n} $, $ f = \oplus_{n=0}^{\infty}f_n \in \FF_T(\HH)$, and $ g = \oplus_{n=0}^{\infty}g_n \in\FF_T(\HH)$.
\end{defn}
We will always use $ \langle \cdot,\cdot\rangle_{T} $ to denote the inner product on $ \FF_T(\HH) $ and $ \HH_{T,n} $, and will denote the untwisted inner product on $\FF_0(\HH)$ and $ \HH^{\otimes n} $ by $ \langle \cdot,\cdot\rangle_0 $ or simply $\langle \cdot,\cdot\rangle$.

By the defining formula $ P_{T,n} = (1\otimes P_{T,n-1})R_{T,n}$, since $ P_{T,n} $ is self-adjoint, we also have $ P_{T,n} = R_{T,n}^*(1\otimes P_{T,n-1}) $, which implies that the left creation operator $$ a^*(h): \Phi_n\mapsto h\otimes \Phi_n,\quad \forall h\in \HH, \Phi_n\in \HH^{\otimes n}$$ sends $ \ker P_{T,n} $ to $ \ker P_{T,n+1} $. Therefore, for each $ h\in \HH $, the left creation operator $ a^*(h) $ is well defined on the dense subset $\HH_{T,n}^0\subseteq \HH_{T,n} $. In fact, one can show that
$$ \|a^*(h)\big|_{\HH_{T,n}}\|_T\leq (1+\|T\|+\cdots +\|T\|^{n})\|h\|.$$
We denote $ a^*_T(h) $ to be its extension to $ \FF_T(\HH) $ to emphasize that it acts on $ \FF_T(\HH) $.

The left annihilation operator $ a_T(h) $ is defined to be the adjoint of $ a^*_T(h) $. On $ \HH_{T,n}^0 $, one can compute that
$$ a_T(h)[ \Phi_n ] = [a(h)R_{T,n}\Phi_n], \quad \forall \Phi_n\in \HH^{\otimes n},$$
where $ a(h) $ is the standard annihilation operator on the full Fock space $ \FF_0(\HH)=\bigoplus_{n=0}^{\infty}\HH^{\otimes n} $: $ a(h): f_1\otimes \cdots \otimes f_n \mapsto \langle h,f_1\rangle f_2\otimes\cdots \otimes f_n$.

We will mostly be working with strict twist $T\in \mathcal{T}_{>}$, and thus $ \text{ker}P_{T,n}={0} $. In this case, we may use notation $ \xi_1\otimes\cdots\otimes\xi_n $ to denote either a vector in $ \HH^{\otimes n} $ or a vector in $ \HH_{T,n} $.

\subsection{Twisted Araki-Woods von Neumann algebra}
Let $ H\subset \HH $ be a standard subspace with Tomita operator $ S_H = J_H\Delta_H^{1/2} $. That is, $ H $ is a closed real subspace of $\HH$ such that $H+iH\subseteq \HH $ is dense, $ S_H $ is the closed operator with domain $D(S_H)=H+iH$ and $$S_H(h+ih') = h-ih',\quad\forall h,h'\in H,$$ which has polar decomposition $ S_H = J_H\Delta_H^{1/2} $. Let $ T $ be a twist for $ \HH $. We consider the left field operators $$ X_T(\xi) := a^*_T(\xi)+a_T(S_H\xi), \quad \forall \xi\in H+iH,$$
which is essentially self-adjoint when $ \|T\|\leq 1 $ and bounded when $ \|T\|<1 $.

\begin{defn}
	Given a standard subspace $ H\subset \HH $ and a twist $ T\in \mathcal{T}_\geq $ for $ \HH $, we define the $T$-twisted Araki-Woods von Neumann algebra as
	$$ \LL_{T}(H):= \{ \exp(iX_T(\xi)):\xi\in H+iH \}''\subseteq B(\FF_T(\HH)). $$
	When $ X_T(\xi) $ are all bounded operators (in particular, when $ \|T\|<1 $), $ \LL_{T}(H) $ is simply the weak-$*$-closure of the $*$-algebra $ \mathcal{P}_{T}(H) $ of all polynomials in $ X_T(\xi) $ for $ \xi\in H $.
\end{defn}

We will always assume that a twist $ T $ is compatible with the embedding $H\subset \HH$ in the following sense.
\begin{defn}
	Let $H\subset \HH$ be a standard subspace with modular operator $\Delta_H$. A twist $ T\in \mathcal{T}_{\geq} $ on $ \HH $ is said to be compatible with $H$ if
	$$ [\Delta_H^{it}\otimes \Delta_H^{it},T] = 0, \quad \forall t\in \RR.$$
	We will denote $ \mathcal{T}_{\geq}(H) $ the set of all such twists, and $ \mathcal{T}_{>}(H):=\mathcal{T}_{\geq}(H)\cap \mathcal{T}_{>} $.
\end{defn}

\subsection{Separablity of the vacuum vector}
For a compatible twist $T\in \mathcal{T}_{\geq}(H)$, \cite{dL22} shows that the vacuum vector $ \Omega $ is cyclic and separating for $\mathcal{L}_T(H)$, if and only if \begin{enumerate}[(1)]
	\item $T$ braided: $T_1T_2T_1=T_2T_1T_2$,
	\item $ T $ is crossing symmetric (defined below).
\end{enumerate}

The crossing symmetric condition is defined using analytic functions over strips. We follow the notation in \cite{dL22}: Let $ \SSS_{1/2} $ be the strip region $ \SSS_{1/2} := \{  z\in \CC: -1/2< \mbox{Im}(z)< 0  \} $. The vector space of bounded analytic functions on $ \SSS_{1/2} $ is denoted
$$ \HHH_{c.bv.}^\infty(\SSS_{1/2}):=\{ f:\overline{\SSS_{-1/2}}\to \CC \mbox{ is continous and bounded, and }f|_{\SSS_{1/2}} \mbox{ is analytic.} \}.$$

\begin{defn}[Crossing Symmetry]
	Let $ T\in \mathcal{T}_{\geq} $ be a twist on a (complex) Hilbert space $\HH$, and $ H\subset \HH $ a standard real Hilbert subspace with the modular operator $ \Delta_H $ and modular conjugation $ J_H $. Then $T$ is said to be crossing symmetric with respect to $ H $ if for all $ \psi_1,\cdots,\psi_4\in \HH $, the function
	\begin{equation}\label{crossing symmetry 1}
		T^{\psi_2,\psi_1}_{\psi_3,\psi_4}(t):= \langle \psi_2\otimes \psi_1,(\Delta_H^{it}\otimes 1)T(1\otimes \Delta_H^{-it})(\psi_3\otimes \psi_4) \rangle
	\end{equation}
	lies in $ \HHH_{c.bv.}^\infty(\SSS_{1/2}) $, and for $t\in \RR$,
		\begin{equation}\label{crossing symmetry 2}
		T^{\psi_2,\psi_1}_{\psi_3,\psi_4}(t+\frac{i}{2}):= \langle \psi_1\otimes J_H\psi_4,(1\otimes \Delta_H^{it})T(\Delta_H^{-it}\otimes 1)(J_H\psi_2\otimes \psi_3) \rangle.
	\end{equation}
\end{defn}
In particular, when $ \Delta_H=\text{id} $ ($H\subseteq \HH$ is tracial), the crossing symmetry becomes the cyclic relations
$$ t^{kl}_{ij}=t^{lj}_{ki},\quad\forall i,j,k,l\in I,$$
where $ t^{kl}_{ij}:= \langle e_{k}\otimes e_{i}, T(e_{i}\otimes e_j)\rangle $ are the coefficients for a fixed orthonomal basis $ \{e_i\}_{i\in I} $ of $H$.

There are also several equivalent forms of the crossing symmetry stated in \cite{dL22}. The idea of the proofs is mostly due to \cite{dL22}, and we sketch the proof here for the sake of remaining self-contained. (See also Corollary \ref{C1T2=C2T1} for another equivalent form using the contraction operators.)
\begin{lem}\label{all equivalent form of CS}
	For standard subspace $H\subset \HH$ and a compatible twist $ T\in \mathcal{T}_{\geq }(H) $, the followings are equivalent.
	\begin{enumerate}[a)]
		\item $T$ is crossing symmetric.
		\item For any $ \psi_2,\varphi_1\in \HH $ and $ \psi_1,\varphi_2\in H+iH $, the function
		$$ f(t) = \langle \psi_1\otimes \psi_2, T(\varphi_1\otimes \Delta_H^{it}\varphi_2)\rangle,\quad t\in \RR $$
		belongs to $ \HHH_{c.bv.}( \SSS_{-1}) $ i.e. $f$ is analytic on the strip $ \SSS_{-1}=\{ z\in \CC: -1<\text{Im}(z)<0 \} $ and continuous on $ \overline{\SSS_{-1}} $. At boundary,
		$$ f(t-i) = \langle \psi_2\otimes S_H \Delta_H^{it}\varphi_2,T(S_H\psi_1 \otimes\varphi_1)\rangle. $$
		\item For any $ \varphi_1\in H+iH $, $ \varphi_2\in \HH $, and $ \varphi_3\in H'+iH' $,
		$$ a(S_H\varphi_1)T(\varphi_2\otimes \varphi_3) =a_r(S_H^*\varphi_3)T(\varphi_1\otimes \varphi_2),$$
		where $ a_r(\xi) $ is the right annihilation operator $ a_r(\xi)(\eta_1\otimes\cdots\otimes \eta_m) = \langle \xi,\eta_m \rangle \eta_1\otimes \cdots\otimes \eta_{m-1} $.
		\item For all $ \varphi_1,\varphi_3\in \HH $, $ \varphi_2\in H+iH $, $ a(\varphi_1)T(\varphi_2\otimes \varphi_3) \in D(S_H) = H+iH $ and
		$$ S_Ha(\varphi_1)T(\varphi_2\otimes \varphi_3) = a(\varphi_3)T(S_H \varphi_2 \otimes \varphi_1). $$
	\end{enumerate}
\end{lem}
\begin{proof}
	b)$\implies$a). This is in the proof of Theorem 3.12 a) \cite{dL22}. Let $ \psi_1,\psi_2,\varphi_1,\varphi_2 $ be as in b). We first assume that $ \psi_1,\varphi_2 $ are analytic. Define the holomorphic function $$ g(t,s):= \langle \Delta_H^{-i\bar{s}}\psi_1\otimes \psi_2,T( \varphi_1\otimes \Delta_H^{-it}\varphi_2 ) \rangle.$$ Then by b), \begin{align*}
		g(t+\frac{i}{2},s)&=g((t-\frac{i}{2})+i,s) = \langle \psi_2\otimes S_H \Delta_H^{-i(t-i/2)}\varphi_2,T( S_H\Delta_H^{-i\bar{s}}\psi_1 \otimes\varphi_1) \rangle\\
		&=\langle \psi_2\otimes \Delta_H^{-it}J_H\varphi_2,T( \Delta_H^{-is-1/2}J_H\psi_1 \otimes\varphi_1) \rangle
	\end{align*}
	Let $ s= t+i/2 $, we obtain $ g(t+\frac{i}{2},t+\frac{i}{2}) = \langle \psi_2\otimes \Delta_H^{-it}J_H\varphi_2,T( \Delta_H^{-it}J_H\psi_1 \otimes\varphi_1) \rangle $ which is precisely the crossing symmetric condition. For general $ \psi_1,\varphi_2 $, we apply the density argument from Lemma 3.7 \cite{dL22}.
	
	a)$\implies$c). This is in the proof of Theorem 3.22 a) \cite{dL22}. For any $ \xi,\varphi_2\in \HH $, $ \varphi_1 \in H+iH $, $ \varphi_3 $ analytic,
	\begin{align*}
		\langle \xi, a(S_H\varphi_1)T(\varphi_2\otimes \varphi_3)\rangle &= \langle S_H\varphi_1 \otimes \xi, T(\varphi_2\otimes \varphi_3) \rangle = \langle \Delta_H^{-1/2}J_H\varphi_1 \otimes \xi, T(\varphi_2\otimes \Delta_H^{1/2} J_H S_H^*\varphi_3) \rangle\\
		(\text{by a)})&=\langle \xi \otimes J_H J_HS^*_H\varphi_3, T(J_HJ_H \varphi_1\otimes \varphi_2) \rangle = \langle \xi, a_r(S^*_H\varphi_3)T( \varphi_1\otimes \varphi_2) \rangle,
	\end{align*}
	hence $ a(S_H\varphi_1)T(\varphi_2\otimes \varphi_3) = a_r(S^*_H\varphi_3)T( \varphi_1\otimes \varphi_2) $. For a general $\varphi_3\in \HH$, we simply note that analytic vectors form a core of $ S_H^* $.
	
	c)$\implies$d). Let $ x\in H+iH $, $ y,w\in \HH $ and $ z\in H'+iH' $, then by c)
	$$ \langle x \otimes w,T(y\otimes z) \rangle = \langle w,a(S_HS_Hx)T(y\otimes z)\rangle = \langle w,a_r(S_H^*z)T(S_Hx\otimes y)\rangle = \langle w\otimes S_H^*z,T(S_Hx\otimes y)\rangle. $$
	On the one hand, $ \langle x \otimes w,T(y\otimes z) \rangle  = \langle a(y)T(x\otimes w),z\rangle $. On the other hand, $ \langle w\otimes S_H^*z,T(S_Hx\otimes y)\rangle = \langle S_H^*z,a(w)T(S_Hx\otimes y)\rangle $. Therefore, we have $ \langle a(y)T(x\otimes w),z\rangle = \langle S_H^*z,a(w)T(S_Hx\otimes y)\rangle $, hence $ a(y)T(x\otimes w)= S_H a(w)T(S_Hx\otimes y)$ by the density of $ H'+iH' $. 
	
	d)$\implies$b). Again by the density argument from Lemma 3.7 \cite{dL22}, it suffices to prove b) for analytic $ \varphi_2 $. Assume $ \psi_1,\varphi_1\in \HH $ and $ \psi_2\in H+iH $ and $ \varphi_2 $ analytic. Then the function $f(t) = \langle \psi_1\otimes \psi_2, T(\varphi_1\otimes \Delta_H^{it}\varphi_2)\rangle$ extends to $\overline{\SSS}_{-1}$ and
	\begin{align*}
		f(t-i) &= \langle \psi_1\otimes \psi_2, T(\varphi_1\otimes \Delta_H^{it}\Delta_H\varphi_2)\rangle = \langle a(\varphi_1)T(\psi_1\otimes \psi_2), \Delta_H\Delta_H^{it}\varphi_2\rangle\\
		&=  \langle S_H\Delta_H^{it}\varphi_2,S_Ha(\varphi_1)T(\psi_1\otimes \psi_2)\rangle\\
		\intertext{applying d)}&= \langle S_H\Delta_H^{it}\varphi_2,a(\psi_2)T(S_H\psi_1\otimes \varphi_1)\rangle = \langle \psi_2\otimes S_H\Delta_H^{it}\varphi_2,T(S_H\psi_1\otimes \varphi_1)\rangle.
	\end{align*}
\end{proof}

\begin{thm}[Corollary 3.23, Proposition 3.25 \cite{dL22}]\label{modular data on frakA}
	Let $H\subseteq \HH$ be a standard subspace, and $ T\in \mathcal{T}_{\geq}(H) $ be a compatible twist. The following are equivalent.
	\begin{enumerate}[a)]
		\item The vacuum vector $\Omega$ is cyclic and separating for $\mathcal{L}_T(H)$.
		\item $T$ is braided and crossing symmetric.
	\end{enumerate}
	Assume one of the above condition holds. Let $ \varphi_\Omega =\langle \Omega,\cdot\Omega\rangle$ be the vacuum state on $ \mathcal{L}_T(H) $. Denote $ \Delta_\Omega $ the modular operator for $ (\LL_{T}(H),\varphi_\Omega) $, $ J_\Omega $ the modular conjugation, and $ S_\Omega = J_\Omega \Delta^{1/2}_{\Omega} $ the involution. Then $ \Delta_\Omega $, $ J_\Omega $ are the second quantization of $ \Delta_H $ and $ J_H $. That is, $ \Delta_\Omega^{it} $  and $ J_\Omega $ are (anti)-unitary operator on $\FF_T(\HH)$ satisfying
	\begin{enumerate}[i)]
		\item $ \Delta_\Omega^{it}[ \xi_1\otimes \cdots\xi_n ] = [ \Delta_H^{it}\xi_1\otimes \cdots \Delta_H^{it}\xi_n ],\quad \forall \xi_i\in \HH $.
		\item $ J_\Omega[ \xi_1\otimes \cdots\xi_n ] = [ J_H\xi_n\otimes \cdots J_H\xi_1 ] ,\quad \forall \xi_i\in \HH $.
	\end{enumerate}
\end{thm}

In particular, this implies that $ \xi\in H $ is an eigenvector of $ \Delta_H $ if and only if $ X_T(\xi) $ is an eigenoperator of $ \varphi_\Omega $.

For the rest of the paper, unless further stated we always assume that
\begin{enumerate}[a)]
	\item $\HH$ is finite dimensional;
	\item The compatible twist $T\in \mathcal{T}_{>}(H)$ is braided and crossing symmetric (in particular the vacuum state $\varphi_\Omega$ is a faithful normal state on $ \mathcal{L}_T(H) $);
	\item $ \|T\|=q<1 $.
\end{enumerate}
In particular, the involution $ S_H $ is now a bounded anti-linear operator on the whole space $\HH = H+iH=H'+iH'$. 

Let $e_1,\cdots,e_d$ be a (not necessarily orthogonal) basis of $\HH$, then $ X_T(e_1),\cdots,X_T(e_d) $ generate the von Neumann algebra $ \mathcal{L}_T(H) $. Our main goal is to study the conjugate system of $ (X_T(e_1),\cdots,X_T(e_d)) $ in $(\mathcal{L}_T(H),\varphi_\Omega)$.

\subsection{Examples of braided and crossing symmetric twists}
For finite dimensional $\HH$, while there are a lot of $T\in B(\HH\otimes \HH)$ satisfying the Yang-Baxter equation, many of them are not crossing symmetric, (for example $ T = \mbox{id}_{\HH\otimes \HH} $). In this subsection, we give some examples of compatible braided and crossing symmetric twists that are not the $q$-scaled tensor flip $qF$.

When $H\subset \HH$ is tracial (i.e. $ \HH $ is the complexification of $H$, or equivalently the spectrum $\text{Sp}(\Delta_H)=\{1\}$), one can find a lot of crossing symmetric ($t^{kl}_{ij}=t^{lj}_{ki}$) twists when $d=\text{dim}\;\HH$ is large (say $d\geq 3$). When $d=2 $, there are precisely four continuous families of crossing symmetric braided twists up to a change of basis (found with the help of computer algebra).
\begin{eg}
	Fix an orthonormal basis $\{e_1,e_2\}$ of $ H $ and denote $ T_{ij}^{kl} = \langle e_{k}\otimes e_{l},T(e_{i}\otimes e_{j}) \rangle $. Then $T$ is self-adjoint and crossing symmetric if and only if
	$$\begin{cases}
		T_{11}^{11}=q_1\in \RR,\quad T_{22}^{22}=q_2\in \RR\\ T_{12}^{21}=T_{21}^{12}=q_{12}\in \RR,\\
		T_{12}^{22}=T_{21}^{22}=T_{22}^{12}=T_{22}^{21}=a\in \RR,\\
		T_{21}^{11}=T_{12}^{11}=T_{11}^{21}=T_{11}^{12}=b\in \RR,\\
		T_{11}^{22}=T_{22}^{11}=T_{12}^{12}=T_{21}^{21}=c\in \RR.
	\end{cases}$$
	If $T$ satisfies the Yang-Baxter equation, then up to a change of basis of $H$, $T$ is of one of the following forms (with the basis $\{e_1\otimes e_1,e_1\otimes e_2,e_2\otimes e_1,e_2\otimes e_2\}$)
	$$ \begin{pmatrix}
		q_1& 0&0&0\\
		0&0&q_{12}&0\\
		0&q_{12}&0&0\\
		0&0&0&q_{2}
	\end{pmatrix}, \begin{pmatrix}
	q_{1}& 0&0&c\\
	0&c&-q_{1}&0\\
	0&-q_{1}&c&0\\
	c&0&0&q_{1}
\end{pmatrix}, \begin{pmatrix}
q_{1}& 0&0&\frac{q_1+q_2}{2}\\
0&\frac{q_1+q_2}{2}&\varepsilon\sqrt{\frac{q_1^2+q_2^2}{2}}&0\\
0&\varepsilon\sqrt{\frac{q_1^2+q_2^2}{2}}&\frac{q_1+q_2}{2}&0\\
\frac{q_1+q_2}{2}&0&0&q_{2}
\end{pmatrix},$$
where $\varepsilon = 1 \text{ or} -1$.
\end{eg}

For nontracial cases, the crossing symmetric condition is stricter, and one should expect less crossing symmetric braided twists. The following is the $q_{ij}$-twist for a finite dimensional (almost periodic) standard subspace $ H\subset \HH $. And $ \mathcal{L}_T(H) $ for $q_{ij}$-twist is also called the mixed $q$-Araki-Woods algebra. Note that in general $ q_{ij} $ are allowed to be complex numbers (unlike in the tracial cases).
\begin{eg}
	Let $ H\subset \HH $ be finite dimensional, and $\{ e_1,\cdots,e_d\} $ is an orthonomal basis of eigenvectors of $ \Delta_H $. We assume that $ \{e_1,\cdots,e_d\} $ is invariant under $ J_H $, so that $ J_H e_i = e_{\bar{i}} $ for some $ 1 \leq \bar{i}\leq d $. Then the $q_{ij}$-twist is the operator
	$$ T(e_i \otimes e_j) = q_{ij} e_j\otimes e_i, \quad \forall 1\leq i,j\leq d. $$
	The self-adjointness of $T$ is equivalent to $ q_{ij} = \overline{q_{ji}}\in \CC $. $ T $ is automatically braided and compatible with $ H $, and $ T $ is crossing symmetric if and only if for all $ 1\leq i,j\leq d $,
	$$ q_{ij} = q_{\bar{j}i} = q_{\bar{i}\bar{j}} = q_{j\bar{i}}\in \CC. $$
	(A special family of examples are when $ q_{ij} = q_{\bar{j}i} = q_{\bar{i}\bar{j}} = q_{j\bar{i}} = q_{ji}=q_{i\bar{j}} = q_{\bar{j}\bar{i}} = q_{\bar{i}j}\in \RR $.)
	In particular, for the tracial cases, by the self-adjointness we always have
	$$ q_{ij} = q_{ji}\in \RR. $$
	In some low dimensional cases (especially when $ \Delta_H $ lack fixed vectors), those are the only possible compatible braided crossing symmetric twists. For example, if $ \HH = 2 $ and the spectrum $ \text{Sp}(\Delta_H)=\{\lambda,\lambda^{-1}\} $ with $\lambda > 1$, then one can easily show that $qF$'s are the only compatible braided crossing symmetric twists. Similarly, more complicated calculations show that if either (a) $ \HH = 3 $ and $ \text{Sp}(\Delta_H) = \{1,\lambda ,\lambda^{-1}\} $ with $ \lambda >1 $, or (b) if $ \HH = 4 $ and $ \text{Sp}(\Delta_H) = \{\lambda_1 ,\lambda_2,\lambda^{-1}_1,\lambda_2^{-1}\}$ with $ \lambda_1>\lambda_2>1 $, then any compatible braided crossing symmetric twist must be a $ q_{ij} $-twist. 
\end{eg}

The next example is borrowed from \cite{dL22}. In fact, a special tracial version of this example also appears in \cite{Kr00} (see also \cite{bikram2021neumann}).
\begin{eg}
	Let $ H\subset \HH $ be a standard subspace, and $$ T = F(A\otimes A), $$
	then $ T $ is a (compatible) braided and crossing symmetric twist only if $A$ is a self-adjoint operator on $ \HH $, and $ [A,\Delta_H^{it}] = [A,J_H] = 0 $. If $H$ is finite dimensional, then we may apply the spectral theorem to $A$ and pick a orthogonal basis of common eigenvectors of $ A $, $\Delta_H$ (invariant under $ J_H $). Under this basis, we just get back to the $q_{ij}$ cases with $ q_{ij} = a_i a_j $ with $ a_i $'s the eigenvalues of $A$.
\end{eg}

\begin{rmk}
	For the braided relation (Yang-Baxter relation) $ T_1 T_2 T_1=T_2 T_1 T_2 $, one can show that $ T $ is braided if and only $ FTF $ is braided. Similarly, for unitary $ U\in B(\HH) $, $ T $ is braided if and only if $ (U\otimes U)T(U^*\otimes U^*) $ is braided. For crossing symmetry, the first claim is no longer true: $ T $ being crossing symmetric does not necessarily imply that $ FTF $ is crossing symmetric. Nevertheless, if $ U $ satisfies $ [U,\Delta_H^{it}] = [U,J_H] = 0 $, then we still have that $ T $ is crossing symmetric if and only if $ (U\otimes U)T(U^*\otimes U^*) $ if crossing symmetric. Therefore, the set of compatible, braided, and crossing symmetric twists is invariant under the conjugation of $ U\otimes U $ for unitary $ U $ commuting with $ \Delta_H^{it}$ and $ J_H $. 
\end{rmk}

So far, all the examples has the property that if $\Lambda= \{\lambda_1,\cdots,\lambda_k\} $ is a subset of eigenvalues of $ \Delta_H $ closed under reciprocal, and $ \KK $ is the direct sum of the eigenspaces of eigenvalues in $\Lambda$ (equivalently $ \KK $ is a $\Delta_H$ and $ J_H $ hyperinvariant subspace), then $ T(\KK \otimes \KK)\subseteq \KK \otimes \KK $. This implies that if we consider the standard subspace $ K :=\KK\cap H\subset \KK $, then $ \mathcal{L}_{T|_{\KK \otimes \KK}}(K) $ is an expected subalgebra of $ \mathcal{L}_{T}(\HH) $. However, this is not true in general, as we have the following example. (See also the remark at the end of Appendix A about this example.)

\begin{eg}\label{twist on matrix}
	Let $ \HH = M_n(\CC)$ with the inner product $\langle x,y\rangle = \text{Tr}(hx^*y)$ where $h\in M_n(\CC)$ is diagonal $ h=\text{diag}\{h_1,\cdots,h_n\}>0 $ and $H = M_n(\CC)_{s.a.}$ be the standard subspace consist of all Hermitian matrices. Let $ m: M_n(\CC)\otimes M_n(\CC) \to M_n(\CC)$ be the multiplication operator $ m(x\otimes y)=xy $. Consider $ m^*: M_n(\CC)\to M_n(\CC)\otimes M_n(\CC) $ the adjoint of $ m $, then we have
	$$ m^*(E_{ij}) = \sum_{k=1}^{n} \frac{1}{h_{k}}E_{ik}\otimes E_{kj}, \quad \forall 1\leq i,j\leq n,$$
	where $ E_{ij} $ is the elementary matrix with the $(i,j)$-th coefficient $1$ and all other coefficients $0$'s. Let $ T := cm^*m: M_n(\CC)\otimes M_n(\CC)\to M_n(\CC)\otimes M_n(\CC) $ for some constant $ c $, then $ T $ is compatible, braided, and crossing symmetric. In fact, $ T= c\text{Tr}(h^{-1})Q_h $ where $Q_h$ is the orthogonal projection onto the subspace $ \text{span}\{ \sum_{k=1}^{n} \frac{1}{h_{k}}E_{ik}\otimes E_{kj}: 1\leq i,j\leq n\} \subseteq M_n(\CC)\otimes M_n(\CC)$. (Therefore $T$ is a twist whenever $ c\geq -1/\text{Tr}(h^{-1})$.) The crossing symmetry of $T$ follows directly from the identity $\langle Y\otimes X,T(Z\otimes W)\rangle= c\langle YX,ZW\rangle $ for all $X,Y,Z,W\in M_n(\CC)$. Indeed,
	\begin{align*}
		&\langle Y\otimes X,(\Delta_H^{i(t+i/2)}\otimes \text{id})T(Z\otimes \Delta_H^{-i(t+i/2)} W)\rangle = c\langle  (\Delta_H^{-i(t-i/2)}Y) X,Z\Delta_H^{-i(t+i/2)} W) \rangle\\
		=&c\langle  X(S_H^*\Delta_H^{-i(t+i/2)} W),(S_H\Delta_H^{-i(t-i/2)}Y) Z \rangle=c\langle X \Delta_H^{-it}J_HW,(\Delta_H^{-it}J_H Y)Z \rangle\\
		=&\langle X \otimes \Delta_H^{-it}J_HW,T(\Delta_H^{-it}J_H Y\otimes Z) \rangle.
	\end{align*}
	To see that $T$ is also braided, we note that $ m^*m = (m\otimes \text{id}) (\text{id}\otimes m^*) = (\text{id}\otimes m)(m^*\otimes \text{id}) $ (this follows from $ (m\otimes \text{id}) (\text{id}\otimes m^*)(X\otimes Y) = X\cdot m^*(Y) = m^*(XY) $), and therefore
	\begin{align*}
		T_1(\text{id}\otimes m^*) &= c(m^*\otimes \text{id})(m\otimes \text{id})(\text{id}\otimes m^*)=c(m^*\otimes \text{id})m^*m\\
		&=c(\text{id}\otimes m^*)m^*m = c(\text{id}\otimes m^*)(\text{id}\otimes m)(m^*\otimes \text{id})\\
		&=T_2(m^*\otimes \text{id}).
	\end{align*}
	And so $ T_1T_2T_1 = c(T_1(\text{id}\otimes m^*))(T_1(\text{id}\otimes m^*))^* =c(T_2(m^*\otimes \text{id}))(T_2(m^*\otimes \text{id}))^* =T_2T_1T_2  $. (Similar calculation shows that in fact $T_1T_2 = T_2T_1$.)
	
	The same computation also works for any left Hilbert algebra $ \frakA $ with bounded multiplication $ m:\frakA\otimes \frakA \to \frakA $ and $ T := cm^*m $. Let $ \HH $ be the closure of $ \frakA $ with $ S_H $ the closure of the involution on $ \frakA $, then $ T,m $ extends to $ \HH\otimes \HH $ and $T$ is compatible braided and crossing symmetric. (In fact, one can show that $\frakA$ with bounded multiplication must be of type $I$. We refer to the standard reference \cite{Takesaki2} to the properties of Hilbert algebras.)
	
	Note that when $c\neq 0$, $\HH = M_n(\CC)$ has no proper $S_H$-invariant subspace $\KK$ such that $ T(\KK\otimes \KK)\subseteq \KK\otimes \KK $. To see this, we note that $ ((\text{Tr}\cdot h)\otimes \text{id})m^* = \text{id} $, hence $ m(\KK\otimes \KK) = ((\text{Tr}\cdot h)\otimes \text{id})T(\KK\otimes \KK) \subseteq  ((\text{Tr}\cdot h)\otimes \text{id})(\KK\otimes \KK) \subseteq \KK $. Thus $ \KK $ is a finite dimensional $C^*$-algebra and must have a unit, hence $ m(\KK\otimes \KK)= \KK $. In particular, we have $ m^*\KK = m^*m(\KK\otimes \KK) = T(\KK\otimes \KK) \subseteq \KK \otimes \KK$, or equivalently $ (1-P_{\KK}\otimes P_{\KK})m^*P_{\KK}=0 $ (with $P_{\KK}$ the orthogonal projection onto $ \KK $). Taking the adjoint, we obtain $ P_{\KK}m(1-P_{\KK}\otimes P_{\KK}) =0$, which implies that $ \KK^\perp $ is an ideal of $ M_n(\CC) $ and so must be either $\{0\}$ or $ M_n(\CC) $. Similarly, if $ \HH = L^2(B(\CC^\infty), \text{Tr}(h\cdot)) = \overline{\frakA}_{\text{Tr}(h\cdot)} $ with $ \text{Tr}(h^{-1})<\infty $, then $ m $ is also bounded. Similar argument shows that if $ \KK\subseteq \HH $ such that $\KK \cap H$ is a standard subspace of $ \KK $, then $ \KK=\{0\} $ or $ \KK=\HH $. (In this case, we still have $m(\KK\otimes \KK)$ being dense in $\KK$, as $ \KK\cap D(S_H) $ is a left Hilbert subalgebra of $ \frakA_{\text{Tr}(h\cdot)} $ hence must be nondegenerate.)
\end{eg}

\subsection{Conjugate system and free Fisher information}

Recall that if $(X_1,\cdots,X_d)$ generates a von Neumann algebra $M$ with faithful normal state $ \varphi $, then the conjugate variables of $ (X_1,\cdots,X_d) $ with respect to the free difference quotient $ \partial_i $'s are vectors $ \Xi_i := \partial_i^*(1\otimes 1)\in L^2(M,\varphi) $ such that for all noncommutative polynomial $p\in \CC\langle x_1,\cdots,x_d\rangle$ on $d$ variables,
$$\langle  1\otimes 1,\partial_ip(X_1,\cdots,X_d)\rangle_{\varphi \otimes \varphi^{op}} = \langle \Xi_i, p(X_1,\cdots,X_d)\rangle_{\varphi}, $$
where $ \partial_i: \CC\langle x_1,\cdots,x_d\rangle\to \CC\langle x_1,\cdots,x_d\rangle\otimes \CC\langle x_1,\cdots,x_d\rangle $ is the free difference quotient $ \partial_i(x_{j_1}\cdots x_{j_k}) := \sum_{m=1}^{k}\delta_{i,j_k}x_{j_1}\cdots x_{j_{m-1}}\otimes x_{j_{m+1}}\cdots x_{j_k} $. (Note that we do not require $ X_i $'s to be self-adjoint. See Definition 3.4 and 3.5 in \cite{Ne17} for details.) The family of vectors $ (\Xi_1,\cdots,\Xi_d) $ is called a conjugate system of $ (X_1,\cdots,X_d) $ with respect to $ \partial_i $'s. The free Fisher information for $ X_1,\cdots,X_d $ is defined to be $ \Phi^*_\varphi(X_1,\cdots,X_n) :=\sum_{m=1}^d\| \Xi_i \|^2_{L^2(M,\varphi)} $.

When we pick $ X_1,\cdots,X_d $ to be self-adjoint and satisfying certain covariance and modular condition, we can also define the quasi-free difference quotients $ \eth_i $'s and the corresponding conjugate variables as in \cite{Ne17}\cite{nelson2015free} (see also \cite{KSW23}). But since our computation is mainly in terms of the free difference quotients, we will postpone the discussion for quasi-free difference quotients until Section $5$. Essentially, free difference quotients and quasi-free difference quotients only differ by an invertible linear transform.

\section{Wick Product and Incomplete Matchings}
For the $ q $-Araki-Woods algebra (i.e. when $ T $ is the scaled tensor flip $ T = qF $), it is known that the Wick product can be written via incomplete matchings (partitions whose blocks are singletons and pairings) (See \cite{effros2003feynman} for the tracial cases):
$$ \varPhi(\xi_1\otimes\cdots\otimes \xi_n) = \sum_{\pi\in \mathcal{P}_{1,2}(n)} (-1)^{|p(\pi)|}q^{\Cr(\pi)} \prod_{\{i,j\}\in p(\pi)}\langle S\xi_{i},\xi_j\rangle \prod_{\{k\}\in s(\pi)}X_{qF}(\xi_k), $$
where $ \Cr(\pi) $ is the crossing number of $ \pi $ when drawn with certain rules. 

In this section, we will generalize this formula to the $ T $-twisted Araki-Woods algebra with a general braided and crossing symmetric twist $ T $ (Theorem \ref{Wick product}):
$$ \varPhi(\xi_1\otimes\cdots\otimes\xi_n) = X\bigg(\sum_{\pi\in \mathcal{P}_{1,2}(n)} (-1)^{|p(\pi)|}W_\pi^T(\xi_1\otimes\cdots\otimes \xi_n)\bigg),$$
where $ W^T_\pi $'s are certain twisted contraction operators with norm bound $ \|W^T_\pi\|\leq C^n \|T\|^{\Cr(\pi)} $, and $ X $ is the linear map $ X(\eta_1\otimes \cdots\otimes \eta_k) = X_T(\eta_1)X_T(\eta_2)\cdots X_T(\eta_k) \in \mathcal{L}_T(H)$ for all $ \eta_i\in \HH $. Before we define those twisted contraction operators $ W^T_\pi $'s, we first need to introduce some convenient notations about incomplete matchings.

\subsection{Incomplete matchings $\mathcal{P}_{1,2}(n)$ and admissible linear orders of pairings}
We recall some basic definitions about partitions. A \textit{partition} $\pi$ of a set $S$ is a collection $ \pi= \{V_1,\cdots,V_s\} $ of disjoint nonempty subsets of $S$ such that $ \cup_{i=1}^{s}V_i = S$. Those $ V_i $'s are called blocks of $\pi$, and the number of blocks in $ \pi $ is denoted by $ |\pi| $. We denote by $ \mathcal{P}(S) $ the set of all partitions of $S$, and when $ S =[n] :=\{1,\cdots,n\}   $, we simply write $ \mathcal{P}([n]) = \mathcal{P}(n) $. A partition $ \pi = \{ V_1,\cdots,V_s\}\in \mathcal{P}(S) $ is called an \emph{incomplete matching} if for all $1\leq i\leq s$, $ |V_i|=1 \mbox{ or } 2 $, i.e. each block of $\pi$ is either a pairing or a singleton. The set of all incomplete matchings of $S$ is denoted by $ \mathcal{P}_{1,2}(S) $ (similarly $ \mathcal{P}_{1,2}([n]) = \mathcal{P}_{1,2}(n) $). We denote by $p(\pi):= \{ V\in \pi: |V| = 2 \}   $ the set of all the pairings of $ \pi $, and $ s(\pi) := \{ V\in \pi: |V|=1 \} $ the set of all singletons of $\pi$. For $ S=\varnothing $, we set $ \mathcal{P}(S) =\mathcal{P}(0) = \{ 1^0\} $ where $ 1^0 := \{\varnothing\} $ is the trivial partition of the empty set $ [0]:=\varnothing $.

If $ f:S\to S' $ is a bijection between two sets, we also denote $ f:\mathcal{P}(S)\to \mathcal{P}(S') $ the bijection that sends $ \pi=\{V_1,\cdots,V_s\}\in \mathcal{P}(S)$ to $ f(\pi):= \{f(V_1),\cdots,f(V_s)\} $. When $ S $ is an ordered set with the order preserving counting map $ c_S:[|S|]\to S $, we always have the bijection $c_S: \mathcal{P}_{1,2}(|S|) \to \mathcal{P}_{1,2}(S) $. But since we will frequently deal with partitions on subsets of $[n]$, we sometimes distinguish between $ \mathcal{P}_{1,2}(S) $ and $ \mathcal{P}_{1,2}(|S|) $, and will mention it whenever we made such an identification.

\begin{defn}
	Let $ \pi \in \mathcal{P}_{1,2}(n) $ be an incomplete matching. The nested (strict) partial order $ < $ on $ p(\pi) $ is defined as: For $ V=\{ i,j \}, U=\{k,l\} \in p(\pi)$, $$ V<U,\quad  \text{ if } i<k<l<j.  $$
	A (strict) linear order $ \varphi $ on $ p(\pi) $ is called admissible, if it is a linear extension of the nested partial order $ < $, that is, if $V<U$ then $ V<_\varphi U $.
	
	The left standard order $ L_\pi = L $ is the admissible linear order: for $ V=\{ i,j \}, U=\{k,l\} \in p(\pi)$ with $ i<j $ and $ k<l $,
	$$ V<_L U,\quad \text{if } i<k. $$
	Similarly, we can define the right standard order
	$$ V<_R U,\quad \text{if } j>l. $$
\end{defn}

\begin{eg}
	Let $ \pi = \{  \{1,4\},\{2,6\},\{3,5\}\}\in \mathcal{P}_{1,2}(6) $, then $ \{2,6\}< \{3,5\} $. There are three admissible orders of $ p(\pi) $: a) the left standard order $ \{1,4\}<_L \{2,6\}<_L\{3,5\} $; b) the right standard order $ \{2,6\}<_R \{3,5\}<_R \{1,4\} $; c) and another order $ \{2,6\}<_\varphi \{1,4\}<_\varphi \{3,5\} $.
\end{eg}

\subsection{Crossing numbers $\Cr(\pi)$ for Wick product}
In the formula of Wick product of $q$-Gaussian algebras, \cite{effros2003feynman} shows that the correct way of counting the crossing numbers (w.r.t. the left standard order $L$) of a pairing $ V=\{i,j\}\in p(\pi) $ of $\pi\in \mathcal{P}_{1,2}(n)$ is
\begin{align*}
	&\Cr_L(V):= (j-i-1) - |\{ U=\{k,l\}\in p(\pi): k<i<l<j\} | 
	\\ =&| \{ \{k\}\in s(\pi): i<k<j \} |+| \{ U=\{k,l\}\in p(\pi):  i<k<j<l \} |+ 2| \{ U\in p(\pi): V<U \} |,
\end{align*}
i.e. the number of elements between $ i $ and $j$ that is not an end point of a pairing $U$ 'crossing' $ V $ from the right of $ V $.
We may generalize this to a general admissible order $ \varphi: p(\pi)\to [|p(\pi)|] $. For $ V=\{i,j\}\in p(\pi) $ with $ i<j $, we define
\begin{align*}
	\Cr_{\varphi}(V) :=& (j-i-1) - | \{U=\{k,l\}\in p(\pi): k<i<l<j, U<_\varphi V \}|\\&-| \{U=\{k,l\}\in p(\pi): i<k<j<l, U<_\varphi V \}|\\
	=& | \{U=\{k,l\}\in p(\pi): k<i<l<j, V<_\varphi U\}|+| \{U=\{k,l\}\in p(\pi): i<k<j<l, V<_\varphi U \}|\\
	&+ | \{ \{k\}\in s(\pi): i<k<j \} | + 2| \{ U\in p(\pi): V<U \} |.
\end{align*} 
The total crossing number $ \Cr(\pi) $ is then the sum over all pairings $ V\in p(\pi) $,
\begin{align*}
	\Cr(\pi) :=& \sum_{V\in p(\pi)}\Cr_\varphi(V)\\
	=& \sum_{\{i,j\}\in p(\pi)}| \{ \{k\}\in s(\pi): i<k<j \} |+2| (U,V): U,V\in p(\pi), V<U |\\ &+| \{ (U,V): U=\{i,j\},V=\{k,l\}\in p(\pi), i<k<j<l \} |,
\end{align*}
which is independent of the choice of the admissible linear order $ \varphi $.

The meaning of those crossing numbers is clear when we represent $ \pi $ and the order $ \varphi $ using diagram drawn with the following rules.

\begin{enumerate}[(1)]
	\item List $\{1,\cdots,n\}$ in a staight line.
	\item Let $ V_1=\{i_1,j_1\}<_\varphi \cdots <_\varphi V_{s}=\{i_s,j_s\} $ be the pairings of $ \pi $ with $ i_m<j_m $. For each $1\leq m\leq s$, $ i_m $ is connected to $ j_m $ with height $m$.
	\item Singletons are drawn with straight lines to the top.
\end{enumerate}

\begin{eg}\label{eg1}
	$ \pi = \{\{1,4\},\{2,10\},\{5,8\},\{7,9\},\{3\},\{6\}\} $ with the left standard order is represented as
\begin{center}
	\crossing{1}{10}{1/4,2/10,5/8,7/9}{3,6}{5}
\end{center}
We have $\Cr_L(\{1,4\}) = 2$, $ \Cr_L(\{2,10\}) = 6 $, $ \Cr_L(\{5,8\})=2 $, and $ \Cr_L(\{7,9\}) = 0 $.

Similarly, if we choose the admissible order $$ \{ 2,10 \}<_\varphi \{5,8\}<_\varphi \{1,4\}<_\varphi \{7,9\},$$
then the diagram is
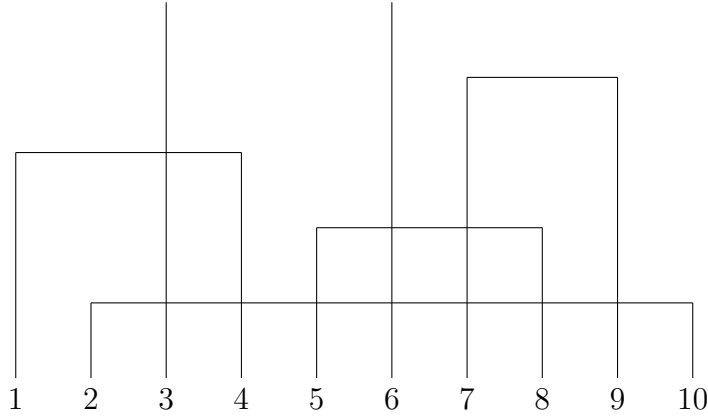
\begin{center}
	\crossing{1}{10}{2/10,5/8,1/4,7/9}{3,6}{5}
\end{center}
And $ \Cr_\varphi(\{2,10\}) = 7 $, $ \Cr_{\varphi}(\{5,8\}) = 2 $, $ \Cr_{\varphi}(\{1,4\})=1 $, and $ \Cr_{\varphi}(\{7,9\})=0 $.
\end{eg}
For a pairing $ V=\{i,j\} \in p(\pi)$, $ \Cr_\varphi(V) $ is the number of crossings on the horizontal part of the chord $\{i,j\}$. And thus $ \Cr(\pi) $ is the total number of crossings in the diagram which is independent of $ \varphi $.

For a fixed order $ \varphi $, let $ V_1=\{i_1,j_1\}<_\varphi \cdots <_\varphi V_{s}=\{i_s,j_s\} $ be the pairings of $ \pi $ with $ i_m<j_m $. Denote $ r_k(m) = r^\varphi_k(m) $ the total number of $ i_t $'s and $ j_t $'s that are less than $ m $ with $ t<k $. Then it is clear from the diagram that the crossing numbers can also be written as
\begin{equation}\label{explicit formula of crossing numbers}
	\Cr_\varphi( V_k ) = (j_k - i_k-1)+ r_k(i_k)-r_k(j_k),
\end{equation}
namely the number of elements between the two end points of $V_k$ that are not the end point of $V_s$ for some $ s<k $. 

\subsection{The twisted contraction operator $ W^T_\pi $}
We now turn to discuss the contraction operators.
\begin{defn}
	Let $ H$ be a standard subspace of a finite dimensional Hilbert space $ \HH $ with involution $S_H$, for each $ k\geq 2 $ and $ 1\leq i<k $, we define the \emph{$i$-th-contraction operator} $ C_i:\HH^{\otimes k}\to \HH^{\otimes k-2} $ as:
	$$ C_i(\xi_1\otimes \cdots\otimes\xi_k)= \langle S_H\xi_i,\xi_{i+1}\rangle\xi_1\otimes \cdots \xi_{i-1}\otimes \xi_{i+2}\otimes\cdots\otimes \xi_k ,\quad \forall \xi_j\in \HH. $$
	$ C_i:\HH^{\otimes k}\to \HH^{\otimes k-2}  $ is a bounded linear operator since $ \HH $ is finite dimensional, and $ \|C_i\|=\|C_1\|\leq (\dim \HH)^{1/2}\cdot \|S_H\| $ (since one have $ C_1^*\Omega = \sum_{i=1}^{d}S_H^*e_i\otimes e_i $ with $ \{e_i\}_{i} $ an orthonormal basis of $\HH$).
\end{defn}
\begin{rmk}
	In general, when $ \HH $ is infinite dimensional, $ C_1 $ is only defined on $ (H+iH)\otimes_{\text{alg}} \HH^{k-1} $.
\end{rmk}

Using the contraction operators, we can rewrite the crossing symmetric condition without indicating the vectors:
\begin{cor}\label{C1T2=C2T1}
	If $T$ is a twist on finite dimensional $\HH$, then $T$ is crossing symmetric if and only if
	$$ C_1T_2 = C_2T_1. $$
	In particular, by taking tensor product with identity on the left, we have $ C_kT_{k+1}=C_{k+1}T_{k} $ for a crossing symmetric twist $T$ and $k\geq 1$.
\end{cor}
\begin{proof}
	For $\xi_1,\cdots,\xi_3\in \HH$, we have
	$$  C_1T_2(\xi_1\otimes \xi_2\otimes \xi_3)= C_1(\xi_1\otimes T(\xi_2\otimes \xi_3))= a(S\xi_1)T(\xi_2\otimes \xi_3). $$
	On the other hand,
	$$ C_2T_1(\xi_1\otimes \xi_2\otimes \xi_3)=C_2(T(\xi_1\otimes \xi_2)\otimes \xi_3). $$
	So by Lemma \ref{all equivalent form of CS}, it suffices to show that $ C_2(T(\xi_1\otimes \xi_2)\otimes \xi_3)=a_r(S^*\xi_3)T(\xi_1\otimes \xi_2) $. For any $\eta\in \HH$,
	\begin{align*}
		\langle \eta, C_2(T(\xi_1\otimes \xi_2)\otimes \xi_3)\rangle &= a(\eta)C_2(T(\xi_1\otimes \xi_2)\otimes \xi_3)=C_1( a(\eta) T(\xi_1\otimes \xi_2)\otimes \xi_3)\\
		&=\langle Sa(\eta)T(\xi_1\otimes \xi_2),\xi_3\rangle = \langle S^*\xi_3,a(\eta)T(\xi_1\otimes \xi_2) \rangle\\
		&=\langle \eta \otimes S^*\xi_3,T(\xi_1\otimes \xi_2)\rangle = \langle \eta, a_r(S^*\xi_3)T(\xi_1\otimes \xi_2)\rangle.
	\end{align*}
	Therefore, we obtain $ C_2(T(\xi_1\otimes \xi_2)\otimes \xi_3)=a_r(S^*\xi_3)T(\xi_1\otimes \xi_2)  $.
\end{proof}

We also need the following relation for braided twist due to \cite{Bo98}. (This can be proven by induction using the Yang-Baxter equation.)
\begin{lem}[\cite{Bo98} Lemma 3]\label{Tst+2Tst+1 = Ts+1t+2Tst+2}
	If $ s\leq t$, then
	\begin{equation}
		(T_{s}T_{s+1}\cdots T_{t+1})(T_{s}T_{s+1}\cdots T_{t})=T_{s,t+2}T_{s,t+1} = T_{s+1,t+2}T_{s,t+2}=(T_{s+1}T_{s+2}\cdots T_{t+1})(T_{s}T_{s+1}\cdots T_{t+1}),
	\end{equation}
	where we used the short notation $ T_{i}T_{i+1}\cdots T_{j-1} = T_{i,j} $ for $ j\geq i+1 $, and $ T_{i,i} = \text{id} $.
\end{lem}

Combine crossing symmetry ($C_1T_2 = C_2T_1$) and the lemma above, we can have more relations between $C_i$ and $ T_i $,
\begin{lem}\label{relations between Ci and Ti}
	Let $ T $ be a braided and crossing symmetric twist.
	\begin{enumerate}[a)]
		\item $ T_sC_t = C_tT_{s+2} $ if $ s\geq t\geq 1 $; $ T_sC_t = C_tT_s $ if $ 1\leq s\leq t-2 $; and $ C_{s}T_{s-1}=C_{s-1}T_{s} $.
		\item For all $k\geq 1$,
		\begin{equation}
			C_1T_{2,k+1}C_{k+1}T_{k+2} = C_kC_1T_{2,k+2}.
		\end{equation}
		In particular, for all $m\geq 1$, $ C_mT_{m+1,k+m}C_{k+m}T_{k+m+1} = C_{k+m-1}C_mT_{m+1,k+m+1} $.
		\item For all $ t<i<k<j $,
		\begin{equation}
			C_{i-1}T_{i,j-2}C_tT_{t+1,k}=C_tT_{t+1,k-1}C_{i}T_{i+1,j}.
		\end{equation}
	\end{enumerate}
\end{lem}
\begin{proof}
	a) The first two equations follow from definition. The third equation is from Lemma \ref{C1T2=C2T1}.\newline
	b) Note that we trivially have $ C_kC_1 = C_1C_{k+2} $ for all $k\geq 1$, hence by a) \begin{align*}
		C_kC_1T_{2,k+2} &= C_1C_{k+2}T_2\cdots T_{k+1} = C_1T_2\cdots T_{k}C_{k+2}T_{k+1}\\ &= C_1T_2\cdots T_{k}C_{k+1}T_{k+2} = C_1T_{2,k+1}C_{k+1}T_{k+2}. 
	\end{align*}\newline
	c) We may assume that $ t=1 $ (for general $t$, we only need to take tensor product with $ \mbox{id}_{\HH}^{\otimes (t-1)} $ on the left). By repeatedly applying a) and b), we have
	\begin{align*}
		C_1T_{2,k-1}C_{i}T_{i+1,j} =& C_{1}T_2\cdots T_{i-1}(T_{i}\cdots T_{k-2})C_iT_{i+1,j}= C_{1}T_{2,i}(T_{i,k-1}C_i)T_{i+1,j}\\
		\stackrel{a)}{=}& C_1T_{2,i}C_i T_{i+2,k+1}T_{i+1,j}=C_1T_{2,i}C_i (T_{i+2,k+1}T_{i+1,k+1})T_{k+1,j}\\
		(\text{Lemma \ref{Tst+2Tst+1 = Ts+1t+2Tst+2}})=& C_1T_{2,i}C_iT_{i+1,k+1}T_{i+1,k}T_{k+1,j}=(C_1T_{2,i}C_iT_{i+1})T_{i+2,k+1}T_{i+1,k}T_{k+1,j}\\
		\stackrel{b)}{=}&C_{i-1}C_1T_{2,i+1}T_{i+2,k+1}(T_{i+1,k}T_{k+1,j})=C_{i-1}C_1T_{2,i+1}T_{i+2,k+1}(T_{k+1,j}T_{i+1,k})\\
		=& C_{i-1}C_1(T_{2,i+1}T_{i+2,j})T_{i+1,k} = C_{i-1}C_1(T_{i+2,j}T_{2,i+1})T_{i+1,k}\\
		=& C_{i-1}C_1T_{i+2,j}T_{2,k} \stackrel{a)}{=} C_{i-1}T_{i,j-2}C_1T_{2,k}.
	\end{align*}
\end{proof}

We can finally define the twisted contraction operator $ W^T_\pi $. (We will first give the formal definition and then illustrate the intuition behind it using diagrams.)
\begin{defn}
	Let $ \pi\in \mathcal{P}_{1,2}(n) $ be a complete matching with an admissible order $ \varphi $ on $p(\pi)$: $$ \{i_1,j_1\}<_\varphi \cdots <_\varphi \{i_s,j_s\},\quad i_m<j_m, \forall 1\leq m\leq s.$$ Then the twisted contraction operator $ W^T_\pi = W^{T}_{\pi,\varphi}: \HH^{\otimes n}\to \HH^{\otimes |s(\pi)|} $ with respect to $ \pi $ is defined as
	\begin{equation}\label{defn of W}
		W^T_{\pi} = C_{i_s-r_s(i_s)}T_{i_s - r_s(i_s)+1,j_s-r_s(j_s)}\cdots C_{i_1-r_1(i_1)}T_{i_1 - r_1(i_1)+1,j_1-r_1(j_1)},
	\end{equation}
	where again we denote $ r_k(m) = r^\varphi_k(m) $ the total number of $ i_t $'s and $ j_t $'s to the left of $ m $ with $ t<k $. When $ p(\pi) = \varnothing $, i.e. $ \pi = \hat{0} $, $ W^T_{\hat{0}} = \text{id}_{\HH}^{\otimes n} $.
	
	In particular, $W^T_\pi$ is the product of $ |p(\pi)| $ contractions and $ \Cr(\pi) $ twists (due to \eqref{explicit formula of crossing numbers}), and $$ \|W^T_\pi\|\leq \|C_1\|^{|p(\pi)|} \|T\|^{\Cr({\pi})}\leq ((\text{dim}\HH)^{1/2}\|S_H\|)^{|p(\pi)|}\|T\|^{\Cr({\pi})}.$$
\end{defn}
Intuitively, we can consider $C_iT_{i+1,j}$ as a kind of twisted contraction from the $i$-th factor to the $j$-th factor in $\xi_1\otimes \cdots \otimes \xi_n$. If $ T=qF $, then $ C_iT_{i+1,j} $ is exactly taking the contraction of $ \xi_i $ and $\xi_j$ (multiplied by $q$), while for general $T$, $C_iT_{i+1,j}$ may involve all $\xi_t$ with $ i\leq t\leq j $. And so $W^{T}_{\pi,\varphi}$ is the operator of consecutively taking twisted contractions according to the order $ \varphi $. For example, if $ \pi = \{ \{1,3\},\{2,5\},\{4,7\},\{6\} \} $ with order $ \{ 2,5 \}<_\varphi \{1,3\}<_\varphi \{4,7\} $, then we have $ W^T_{\pi,\varphi} = C_1T_2C_1C_2T_3T_4 $ as illustrated below.

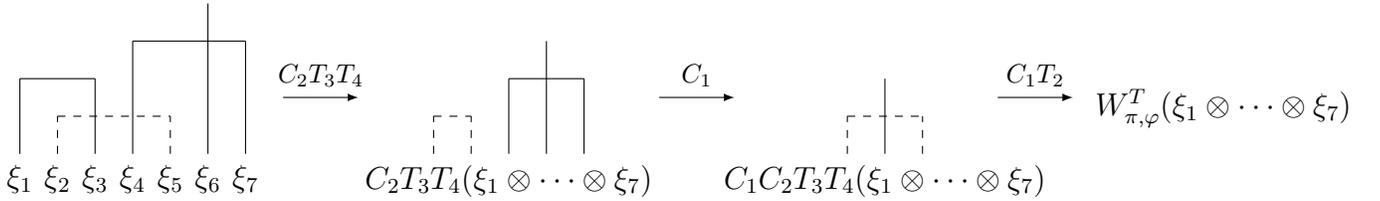
\begin{center}
	\begin{tikzpicture}[scale=0.5]
		\foreach \x in {1,...,7}
		\coordinate (\x) at (\x,0) node [below] at (\x) {$\xi_{\x}$};
		\foreach \x/\y [count=\n] in {2/5,1/3,4/7}{
			\ifthenelse{\n = 1}{   \draw[dashed] (\x)--(\x,\n);
				\draw[dashed] (\y)--(\y, \n);
				\draw[dashed] (\x,\n)--(\y, \n);}
			{   \draw (\x)--(\x,\n);
				\draw (\y)--(\y, \n);
				\draw (\x,\n)--(\y, \n);} };
		\foreach \x in {6}
		{\draw (\x)--(\x, 4);}
		
		\draw[-latex] (8,1.5) -- (10,1.5) node[midway,above] {\footnotesize$C_2T_3T_4$};
		
		\foreach \x in {1,2,...,5}
		\coordinate (\x) at (\x +11,0);
		\foreach \x/\y [count=\n] in {1/2,3/5}{
		  \ifthenelse{\n = 1}{   \draw[dashed] (\x)--(\x+11,\n);
				\draw[dashed] (\y)--(\y+11, \n);
				\draw[dashed] (\x+11,\n)--(\y+11, \n);}
			{\draw (\x)--(\x + 11,\n);
			\draw (\y)--(\y + 11, \n);
			\draw (\x + 11,\n)--(\y + 11, \n);} };
		\foreach \x in {4}
		{\draw (\x)--(\x + 11, 3);}
		
		\coordinate (a) at (14,0) node [below] at (a) {$C_2T_3T_4(\xi_1\otimes \cdots \otimes \xi_{7})$};
		\draw[-latex] (18,1.5) -- (20,1.5) node[midway,above] {\footnotesize$C_1$};
		
		\foreach \x in {1,2,3}
		\coordinate (\x) at (\x +22,0);
		\foreach \x/\y [count=\n] in {1/3}{
			   \draw[dashed] (\x)--(\x+22,\n);
				\draw[dashed] (\y)--(\y+22, \n);
				\draw[dashed] (\x+22,\n)--(\y+22, \n);}
		\foreach \x in {2}
		{\draw (\x)--(\x + 22, 2);}
		
		\coordinate (a) at (24,0) node [below] at (a) {$C_1C_2T_3T_4(\xi_1\otimes \cdots \otimes \xi_{7})$};
		\draw[-latex] (27,1.5) -- (29,1.5) node[midway,above] {\footnotesize$C_1T_2$};
		\coordinate (a) at (33,2) node [below] at (a) {$W^T_{\pi,\varphi}(\xi_1\otimes \cdots \otimes \xi_{7})$};
	\end{tikzpicture}
\end{center}
From the diagram, we can see that the number $ r_k(i_k) $ (reps. $ r_k(j_k) $) in the definition of $ W^T_{\pi,\varphi} $ is nothing but the number of indices to the left of $ i_k $ (reps. $j_k$) that are already 'contracted' before we take the $k$-th contraction.

In fact, the definition of $ W^T_\pi $ does not depend on the choice of the admissible order $ \varphi $. Indeed, Lemma \ref{relations between Ci and Ti} a) and c) implies that permutating two neighboring elements (which are $<$ incomparable) of the order $ \varphi $ will not change the operator $ W^T_{\pi,\varphi} $.

\begin{prop}\label{W is independent of order}
	If $ \varphi,\psi $ are two admissible order on $ p(\pi) $, then $ W^{T}_{\pi,\varphi} = W^{T}_{\pi,\psi} $.
\end{prop}
\begin{proof}
	1) We first show that we can permute two neighboring $<$ incomparable elements of $ \varphi $ without changing the operator. Assume again,
	$$ \{i_1,j_1\}<_\varphi \cdots <_\varphi \{i_s,j_s\},\quad i_m<j_m, \forall 1\leq m\leq s.$$
	Assume that $ \{i_m,j_m\} $ and $ \{i_{m+1},i_{m+1}\} $ are not $<$ incomparable, and that $ \psi $ is the admissible order
	$$    \{i_1,j_1\}<_\psi \cdots <_\psi \{i_{m-1},j_{m-1}\} <_\psi \{i_{m+1},j_{m+1}\}<_\psi \{i_{m},j_{m}\} <_\psi \{i_{m+2},j_{m+2}\} <_\psi\cdots <_\psi \{i_{s},j_{s}\}.$$
	To show that $ W^T_{\pi,\varphi} = W^T_{\pi,\psi} $, by \eqref{defn of W}, it suffices to show that
	\begin{align}\label{Prop 2.9 eq}
		\begin{split}
			&C_{i_{m+1}-r^\varphi_{m+1}(i_{m+1})}T_{i_{m+1} - r^{\varphi}_{m+1}(i_{m+1})+1,j_{m+1}-r^{\varphi}_{m+1}(j_{m+1})} C_{i_{m}-r^\varphi_{m}(i_{m})}T_{i_{m} - r^{\varphi}_{m}(i_{m})+1,j_{m}-r^{\varphi}_{m}(j_{m})}\\
			=&C_{i_{m}-r^\psi_{m+1}(i_{m})}T_{i_{m} - r^{\psi}_{m+1}(i_{m})+1,j_{m}-r^{\psi}_{m+1}(j_{m})}C_{i_{m+1}-r^\psi_{m}(i_{m+1})}T_{i_{m+1} - r^{\psi}_{m+1}(i_{m+1})+1,j_{m+1}-r^{\psi}_{m}(j_{m+1})}.
		\end{split}
	\end{align} 
	 Without loss of generality, we may assume $ i_m<i_{m+1} $ (otherwise we can switch the places of $\varphi$ and $\psi$). Since $ \{i_m,j_m\} $ and $ \{i_{m+1},i_{m+1}\} $ are not $ < $ comparable, either $ i_m<i_{m+1}<j_{m}<j_{m+1} $ or $ i_{m}<j_{m}<i_{m+1}<j_{m+1} $.
	 \begin{enumerate}[(a)]
	 	\item If $ i_m<i_{m+1}<j_{m}<j_{m+1} $, then equation \eqref{Prop 2.9 eq} follows from Lemma \ref{relations between Ci and Ti} c) since $ r^{\psi}_{m+1}(i_{m}) = r^\varphi_{m}(i_{m}) $, $ r^{\psi}_{m+1}(j_{m}) = r^\varphi_{m}(j_{m})+1 $,  $ r^{\psi}_{m}(i_{m+1}) = r^\varphi_{m+1}(i_{m+1})-1 $, and $ r^{\psi}_{m}(j_{m+1}) = r^\varphi_{m+1}(j_{m+1})-2 $.
	 	\begin{center}
	 		\begin{tikzpicture}[scale=0.8]
	 			\coordinate (1) at (1,0) node [below] at (1) {$i_m$};
	 			\coordinate (2) at (2,0) node [below] at (2) {$i_{m+1}$};
	 			\coordinate (3) at (3,0) node [below] at (3) {$j_m$};
	 			\coordinate (4) at (4,0);
	 			\coordinate (5) at (5,0) node [below] at (5) {$j_{m+1}$};
	 			\foreach \x/\y [count=\n] in {1/3,2/5}
	 			{   \draw (\x)--(\x,\n);
	 				\draw (\y)--(\y, \n);
	 				\draw (\x,\n)--(\y, \n);}
	 			
	 			\coordinate (c1) at (3,-0.8) node [below] at (c1) {$W^T_{\pi,\varphi}$};
	 			
	 			\draw[-latex] (6,1) -- (8,1);
	 			
				\coordinate (9) at (9,0) node [below] at (9) {$i_m$};
				\coordinate (10) at (10,0) node [below] at (10) {$i_{m+1}$};
				\coordinate (11) at (11,0) node [below] at (11) {$j_m$};
				\coordinate (12) at (12,0);
				\coordinate (13) at (13,0) node [below] at (13) {$j_{m+1}$};
	 			\foreach \x/\y [count=\n] in {10/13,9/11}
	 			{   \draw (\x)--(\x,\n);
	 				\draw (\y)--(\y, \n);
	 				\draw (\x,\n)--(\y, \n);}
	 			
	 			\coordinate (c2) at (11,-0.8) node [below] at (c2) {$ W^T_{\pi,\psi} $};
	 		\end{tikzpicture}
	 	\end{center}
	 	\item If $ i_{m}<j_{m}<i_{m+1}<j_{m+1} $, then similarly equation \eqref{Prop 2.9 eq} follows from Lemma \ref{relations between Ci and Ti} a) since $ r^{\psi}_{m+1}(i_{m}) = r^\varphi_{m}(i_{m}) $, $ r^{\psi}_{m+1}(j_{m}) = r^\varphi_{m}(j_{m}) $,  $ r^{\psi}_{m}(i_{m+1}) = r^\varphi_{m+1}(i_{m+1})-2 $, and $ r^{\psi}_{m}(j_{m+1}) = r^\varphi_{m+1}(j_{m+1})-2 $.
	 		 	\begin{center}
	 		\begin{tikzpicture}[scale=0.8]
	 			\coordinate (1) at (1,0) node [below] at (1) {$i_m$};
	 			\coordinate (3) at (3,0) node [below] at (3) {$i_{m+1}$};
	 			\coordinate (2) at (2,0) node [below] at (2) {$j_m$};
	 			\coordinate (4) at (4,0);
	 			\coordinate (5) at (5,0) node [below] at (5) {$j_{m+1}$};
	 			\foreach \x/\y [count=\n] in {1/2,3/5}
	 			{   \draw (\x)--(\x,\n);
	 				\draw (\y)--(\y, \n);
	 				\draw (\x,\n)--(\y, \n);}
	 			
	 			\coordinate (c1) at (3,-0.8) node [below] at (c1) {$W^T_{\pi,\varphi}$};
	 			
	 			\draw[-latex] (6,1) -- (8,1);
	 			
	 			\coordinate (9) at (9,0) node [below] at (9) {$i_m$};
	 			\coordinate (11) at (11,0) node [below] at (11) {$i_{m+1}$};
	 			\coordinate (10) at (10,0) node [below] at (10) {$j_m$};
	 			\coordinate (12) at (12,0);
	 			\coordinate (13) at (13,0) node [below] at (13) {$j_{m+1}$};
	 			\foreach \x/\y [count=\n] in {11/13,9/10}
	 			{   \draw (\x)--(\x,\n);
	 				\draw (\y)--(\y, \n);
	 				\draw (\x,\n)--(\y, \n);}
	 			
	 			\coordinate (c2) at (11,-0.8) node [below] at (c2) {$ W^T_{\pi,\psi} $};
	 		\end{tikzpicture}
	 	\end{center}
	 \end{enumerate}
 2) For two general admissible orders $ \varphi $ and $ \psi $, we can proceed like the bubble sort algorithm: Assume $ V_1<_\varphi V_2<_\varphi\cdots <_\varphi V_s $, if $ V_k>_\psi V_{k+1} $ for some $k$, then we apply 1) to permute them. Since after each step, the inversion number $ |\{ (i,j):   V_1<_\varphi V_2,  V_1>_\psi V_2 \}| $ will decrease, this process must end within finite steps. Therefore, by 1), $ W^T_{\pi,\varphi} = W^T_{\pi,\psi} $.
\end{proof}

\begin{eg}
	Let $ \pi = \{\{1,4\},\{2,10\},\{5,8\},\{7,9\},\{3\},\{6\}\} $, and $\varphi$ be as in Example \ref{eg1}. Then $ W^T_{\pi,L} = C_3 C_2 T_{3,5}C_1 T_{2,8}C_1 T_{2,4} = W^T_{\pi,R} = C_1 T_2 C_4 T_{5} C_6 T_{7,9} C_2 T_{3,10} = W^T_{\pi,\varphi} =  C_3C_1T_2C_4 T_{5,7}C_2 T_{3,10} $.
\end{eg}

\subsection{$T$-Wick product}
 For all $\xi_1,\cdots,\xi_n\in \HH$, since $ \Omega $ is separating for $ \mathcal{L}_{T}(H) $, there is a unique operator $\varPhi(\xi_1\otimes\cdots \otimes\xi_n)\in \mathcal{L}_{T}(H) $ such that $$\varPhi(\xi_1\otimes \cdots \otimes\xi_n)\Omega =  \xi_1\otimes \cdots \otimes \xi_n \in \FF_T(\HH)=L^2(\mathcal{L}_{T}(H),\varphi_\Omega) $$
We call such a $ \varPhi(\xi_1\otimes\cdots \otimes\xi_n) $ the $T$-Wick product (also $T$-Wick polynomials when a basis of $\HH$ is given) of the operators $ X_T(\xi_1),\cdots,X_T(\xi_n) $. Formally, $ \varPhi $ is just inverse of the inclusion $ \mathcal{L}_{T}(H) \ni A \mapsto A\Omega \in \FF_T(\HH) = L^2( \mathcal{L}_{T}(H) ) $. Since $ \xi_1\otimes\cdots\otimes \xi_n = X_T(\xi_1)(\xi_2\otimes\cdots\otimes \xi_n) - a_T(S\xi_1)\xi_2\otimes\cdots\otimes \xi_n $, the $T$-Wick product satisfies the recurrence relation:
\begin{equation}\label{recurrence of Wick product}
	\varPhi( \xi_1\otimes\cdots\otimes \xi_n ) = X_T(\xi_1)\varPhi(\xi_2\otimes\cdots\otimes \xi_n)- \varPhi( a_T(S\xi_1)\xi_2\otimes\cdots\otimes \xi_n ).
\end{equation}

In particular, when $\HH$ is finite dimensional with a fixed basis $ (e_1,\cdots,e_d) $, $ \varPhi(\xi_1\otimes\cdots \otimes\xi_n) $ can always be written as a polynomial in $ X_T(e_1),\cdots,X_T(e_d) $.

For a general twist $T$, unlike the $q$-Gaussian algebras, the formula of $T$-Wick product is less transparent since the dependence of $ \varPhi(\xi_1\otimes\cdots \otimes\xi_n) $ on the factors $ \xi_i $'s is implicit due to the twist $T$. For example, for $n=3$, $$ \varPhi(\xi_1\otimes\xi_2\otimes \xi_3) = X_T(\xi_1)X_T(\xi_2)X_T(\xi_3) - \langle S\xi_1,\xi_2\rangle X_T(\xi_3) - \langle S\xi_2,\xi_3\rangle X_T(\xi_1) - X_T(a(S\xi_1)T(\xi_2\otimes \xi_3)),$$
where $ a(S\xi_1)T(\xi_2\otimes \xi_3) $ may no longer be in the span of $\xi_1,\cdots,\xi_3$ (in the $q$-Araki-Woods case, it is $ q\langle S\xi_1,\xi_3\rangle \xi_2 $).
In \cite{Kr00}, a formula of $T$-Wick product is given in terms of operators of the form $ a^*_T(e_{i_{1}})\cdots a^*_T(e_{i_{s}})a_T(e_{j_{1}})\cdots a_T(e_{j_{t}}) $. We will instead give a different formula in terms of the twisted contractions $ W^T_\pi $ (Theorem \ref{Wick product}).

Since we will prove our $ T $-Wick product formula by induction, we need to discuss the recurrence relations of $ W^T_\pi $.

First, we introduce a counting lemma about $ \mathcal{P}_{1,2}(n) $.
\begin{lem}\label{counting of P_12}
	Let $ \mathcal{P}'_{1,2}(n):= \{ \pi\in \mathcal{P}_{1,2}(n): \{1\}\in s(\pi) \} $ and $ \mathcal{P}''_{1,2}(n): = \{ \pi\in \mathcal{P}_{1,2}(n): \{1\}\notin s(\pi) \} $. For a $\pi\in \mathcal{P}''_{1,2}(n)$, we denote by $ j_{1}^\pi\in \{2,\cdots,n\} $ the element such that $ \{1,j_{1}^\pi\}\in p(\pi) $. Then we have the following bijections:
	\begin{enumerate}[a)]
		\item For each $ n\geq 2 $, $ d_{s}:\mathcal{P}'_{1,2}(n)\to \mathcal{P}_{1,2}(n-1)$, for all $ \pi \in \mathcal{P}'_{1,2}(n)$, $ d_{s}(\pi) $ is the partition obtained by deleting $ \{1\} $ from $ \pi $ and then identifying $ \{2,\cdots,n\} $ with $ [n-1] $ (with the order preserved).
		\item For each $ n\geq 3 $, $D_p: \mathcal{P}''_{1,2}(n)\to \mathcal{P}_{1,2}(n-2)\times [n-1]$, for all $ \pi\in \mathcal{P}''_{1,2}(n) $, $ D_p(\pi) := (d_p(\pi),j_{1}^\pi -1) $ where $ d_p(\pi) $ is the partition obtained by deleting $ \{1,j_1^\pi\} $ from $ \pi $ and then identifying $ [n]/\{1,j_1^\pi\} $ with $ [n-2] $ (with the order preserved).
	\end{enumerate}
\end{lem}
\begin{proof}
	It suffices to construct the inverse of those maps. a) For a $ \sigma\in \mathcal{P}_{1,2}(n-1) $, $ d^{-1}_s(\sigma) $ is simply obtained by right shifting the blocks of $ \sigma $ by $ 1 $ and then adding the singleton $ \{1\} $. b) For a $ \sigma\in \mathcal{P}_{1,2}(n-2) $ and $ 1\leq k\leq n-1 $, let $ f: [n-2]\to [n]/\{1,k+1\} $ be the order preserving bijection. Then $ D_p^{-1}(\sigma,k) $ is the union $ \{\{1,k+1\}\}\cup f(\sigma) $.
\end{proof}

Using \eqref{defn of W}, considering the left standard order, by writing $ r^{d_s(\pi)}_{k}(m) $ and $ r^{d_p(\pi)}_{k}(m) $ in terms of $ r^{\pi}_{k}(m) $, one can easily obtain the following recurrence relations.
\begin{lem}\label{induction on W}
	For each $ \pi\in \mathcal{P}_{1,2}(n) $.
	\begin{enumerate}[a)]
		\item For $ n\geq 2 $, and $ \pi\in \mathcal{P}'_{1,2}(n) $,
		$$ W^T_{\pi}= \mbox{id}_\HH \otimes W_{d_s(\pi)}^T.$$
		Or equivalently,
		$$ W^T_{\pi}(\xi_1\otimes \cdots \otimes \xi_n)= a^*(\xi_1)W^T_{d_s(\pi)}(\xi_2\otimes \cdots \otimes \xi_n). $$
		\item For $ n\geq 3 $, and $ \pi \in \mathcal{P}''_{1,2}(n) $,
		$$ W^T_{\pi}=W^T_{d_p(\pi)} C_1T_{2,j_{1}^\pi}. $$
		Or equivalently,
		$$ W^T_{\pi}(\xi_1\otimes\cdots\otimes \xi_n)= W^T_{d_p(\pi)}( a(S\xi_1)T_{1,j_{1}^\pi-1}(\xi_2\otimes\cdots\otimes \xi_n) ).$$
	\end{enumerate}
\end{lem}
\begin{proof}
	Considering the left standard order on those partitions. We use the same notation as in \eqref{defn of W}.\newline
	a) This follows directly from that $ r^{d_s(\pi)}_{k}(m) = r^{\pi}_{k}(m+1) $ for all $k\geq 1$ and $m\geq 1$. b) One can directly check that $ i^\pi_{m+1} -r^{\pi}_{m+1}(i^\pi_{m+1}) =  i^{d_p(\pi)}_m -r^{d_p(\pi)}_{m}(i^{d_p(\pi)}_m) $ for all $ m\geq 1 $. Therefore by \eqref{defn of W}, $$ W^T_{\pi,L} = W^{T}_{d_p(\pi),L} C_{1-r^\pi_1(1)}T_{1 - r^{\pi}_1(1)+1,j^\pi_1-r^\pi_1(j^\pi_1)} = W^T_{d_p(\pi)} C_1T_{2,j^\pi_1}.$$
\end{proof}

The recurrence relation a) and b) can be intuitively illustrated by the following graphs using the left standard order:\newline
For a), let $ \pi_0 = \{\{1,3\},\{2,5\},\{4\}\} $.
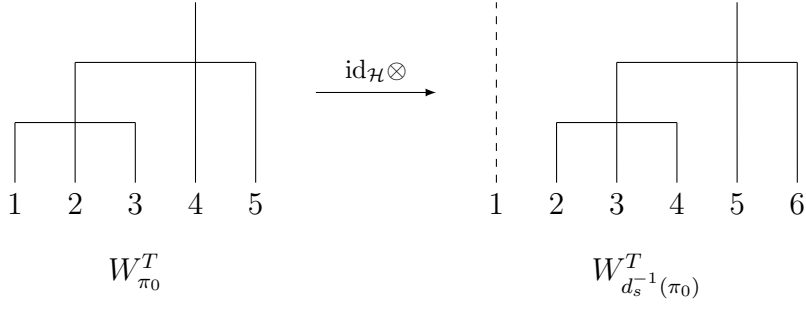
\begin{center}
	\begin{tikzpicture}[scale=0.8]
		\foreach \x in {1,...,5}
		\coordinate (\x) at (\x,0) node [below] at (\x) {$\x$};
		\foreach \x/\y [count=\n] in {1/3,2/5}
		{   \draw (\x)--(\x,\n);
			\draw (\y)--(\y, \n);
			\draw (\x,\n)--(\y, \n);}
		\foreach \x in {4}
		\draw (\x)--(\x, 3);
		
		\coordinate (c1) at (3,-1) node [below] at (c1) {$W^T_{\pi_0}$};
		
		\draw[-latex] (6,1.5) -- (8,1.5) node[midway,above] {\footnotesize$\text{id}_{\HH}\otimes$};
		
		\foreach \x in {1,...,6}
		\coordinate (\x) at (\x + 8,0) node [below] at (\x) {$\x$};
		\foreach \x/\y [count=\n] in {2/4,3/6}
		{   \draw (\x)--(\x + 8,\n);
			\draw (\y)--(\y + 8, \n);
			\draw (\x + 8,\n)--(\y + 8, \n);}
		\foreach \x in {5}
		\draw (\x)--(\x + 8, 3);
		\draw[dashed] (1)--(1 + 8, 3);
		
		\coordinate (c2) at (11.5,-1) node [below] at (c2) {$ W^T_{d_s^{-1}(\pi_0)} $};
	\end{tikzpicture}
\end{center}
For b), let $ \pi = \{\{1,4\},\{2,10\},\{3\},\{5,8\},\{6\},\{7,9\}\} $.
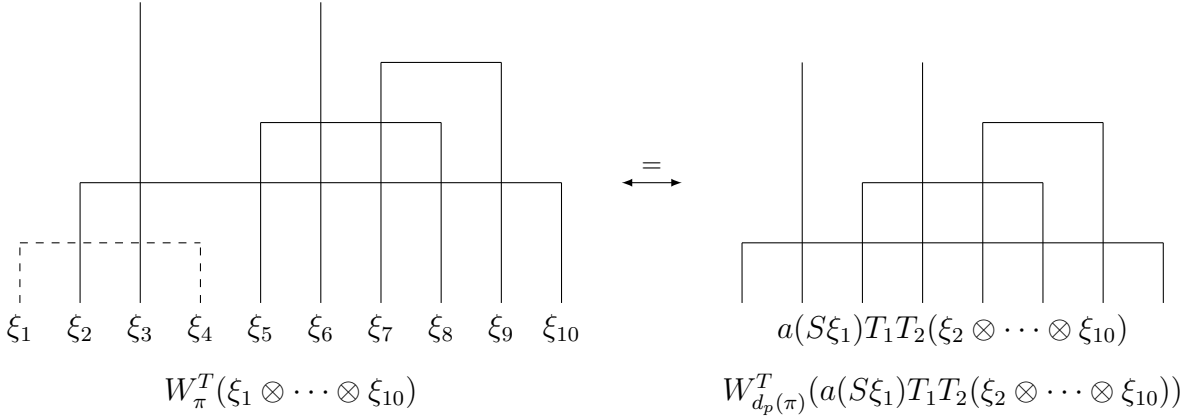
\begin{center}
	\begin{tikzpicture}[scale=0.8]
		\foreach \x in {1,...,10}
		\coordinate (\x) at (\x,0) node [below] at (\x) {$\xi_{\x}$};
		\foreach \x/\y [count=\n] in {1/4,2/10,5/8,7/9}{
			\ifthenelse{\n = 1}{   \draw[dashed] (\x)--(\x,\n);
				\draw[dashed] (\y)--(\y, \n);
				\draw[dashed] (\x,\n)--(\y, \n);}
			{   \draw (\x)--(\x,\n);
				\draw (\y)--(\y, \n);
				\draw (\x,\n)--(\y, \n);} };
		\foreach \x in {3,6}
		{\draw (\x)--(\x, 5);}
		
		\coordinate (c1) at (5.5,-1) node [below] at (c1) {$W^T_\pi(\xi_1\otimes \cdots\otimes \xi_{10})$};
		
		\draw[latex-latex] (11,2) -- (12,2) node[midway,above] {\footnotesize$=$};
		
		\foreach \x in {1,2,...,8}
		\coordinate (\x) at (\x + 12,0);
		\foreach \x/\y [count=\n] in {1/8,3/6,5/7}
		{   \draw (\x)--(\x + 12,\n);
			\draw (\y)--(\y + 12, \n);
			\draw (\x + 12,\n)--(\y + 12, \n);}
		\foreach \x in {2,4}
		{\draw (\x)--(\x + 12, 4);}
		
		\coordinate (a) at (16.5,0) node [below] at (a) {$a(S\xi_1)T_1T_2(\xi_2\otimes \cdots \otimes \xi_{10})$};
		\coordinate (c2) at (16.5,-1) node [below] at (c2) {$W^T_{d_p(\pi)}(a(S\xi_1)T_1T_2(\xi_2\otimes \cdots\otimes \xi_{10}))$};
	\end{tikzpicture}
\end{center}

\begin{thm}\label{Wick product}
	For all $ \xi_1,\cdots,\xi_n \in \HH$,
	$$ \varPhi(\xi_1\otimes\cdots\otimes\xi_n) = X\bigg(\sum_{\pi\in \mathcal{P}_{1,2}(n)} (-1)^{|p(\pi)|}W_\pi^T(\xi_1\otimes\cdots\otimes \xi_n)\bigg),$$
	where $ X $ is the linear map $ X(\eta_1\otimes \cdots\otimes \eta_k) = X_T(\eta_1)X_T(\eta_2)\cdots X_T(\eta_k) \in \mathcal{L}_T(H)$ for all $ \eta_i\in \HH $ and $ X(\Omega) = 1 $.
\end{thm}
\begin{proof}
	We prove by induction on $n$. For $n=1$, $ \varPhi(\xi_1) = X_T(\xi_1) = X(\xi_1) $. And for $ n=2 $, $ \varPhi(\xi_1\otimes\xi_2) = X_T( \xi_1 )X_T(\xi_2)-\langle S\xi_1,\xi_2 \rangle 1 = X( \xi_1\otimes\xi_2 - C_1( \xi_1\otimes\xi_2) ) $. Now assume that the formula holds for vectors with less than $n $ factors with $n>2$. Then by the recurrence relation \eqref{recurrence of Wick product} for the Wick product,
	\begin{align*}
		\varPhi( \xi_1\otimes\cdots\otimes \xi_n ) =& X_T(\xi_1)\varPhi(\xi_2\otimes\cdots\otimes \xi_n)- \varPhi( a_T(S\xi_1)\xi_2\otimes\cdots\otimes \xi_n )\\
		=&X_T(\xi_1)\varPhi(\xi_2\otimes\cdots\otimes \xi_n)- \varPhi( C_1\sum_{k=1}^{n-1}T_{2,k+1}\xi_1\otimes\cdots\otimes \xi_n )\\
		\intertext{by induction hypothesis}=&X_T(\xi_1)X\bigg(\sum_{\sigma\in \mathcal{P}_{1,2}(n-1)} (-1)^{|p(\sigma)|}W_\sigma^T(\xi_2\otimes\cdots\otimes \xi_n)\bigg)\\&- X\bigg(\sum_{\sigma'\in \mathcal{P}_{1,2}(n-2)}\sum_{k=1}^{n-1} (-1)^{|p(\sigma')|}W_{\sigma'}^T C_1 T_{2,k+1}(\xi_1\otimes\cdots\otimes \xi_n)\bigg)\\
		\intertext{applying Lemma \ref{counting of P_12}}=&  \sum_{\pi \in \mathcal{P}'_{1,2}(n)}(-1)^{|p(\pi)|}X_T(\xi_1)XW_{d_s(\pi)}^T(\xi_2\otimes\cdots\otimes \xi_n) \\&- \sum_{\pi \in \mathcal{P}''_{1,2}(n)}(-1)^{|p(d_p(\pi))|}XW_{d_p(\pi)}^T C_1 T_{2,j_1^\pi}(\xi_1\otimes\cdots\otimes \xi_n)\\
		=&  \sum_{\pi \in \mathcal{P}'_{1,2}(n)}(-1)^{|p(\pi)|}X\bigg(\xi_1 \otimes W_{d_s(\pi)}^T(\xi_2\otimes\cdots\otimes \xi_n)\bigg) \\&- \sum_{\pi \in \mathcal{P}''_{1,2}(n)}(-1)^{|p(\pi)|-1}XW_{d_p(\pi)}^T C_1 T_{2,j_1^\pi}(\xi_1\otimes\cdots\otimes \xi_n)\\
		\text{(Lemma \ref{induction on W})} =& \sum_{\pi \in \mathcal{P}'_{1,2}(n)}(-1)^{|p(\pi)|}X W_{\pi}^T(\xi_1\otimes\cdots\otimes \xi_n) + \sum_{\pi \in \mathcal{P}''_{1,2}(n)}(-1)^{|p(\pi)|}XW_{\pi}^T(\xi_1\otimes\cdots\otimes \xi_n)\\
		=& \sum_{\pi \in \mathcal{P}_{1,2}(n)}(-1)^{|p(\pi)|}X W_{\pi}^T(\xi_1\otimes\cdots\otimes \xi_n).
	\end{align*}
\end{proof}
\begin{rmk}
	We note that this formula only works for finite dimensional $ \HH $. If $ \HH $ is infinite dimensional, since $ C_1 $ is unbounded, $W^T_\pi(\xi_1\otimes \cdots \otimes \xi_n)$ or $ XW^T_\pi(\xi_1\otimes \cdots \otimes \xi_n) $ are not well defined for general $\pi$.
\end{rmk}
\subsection{A decomposition of $W_\pi^T$}
\begin{defn} \label{decomposition of P12}
	For a pair $ (\pi,k) $ with $ \pi \in \mathcal{P}_{1,2}(n)$ and $ {k}\in s(\pi) $, we decompose $ \pi $ with respect to $ k $ in the following sense: We divide $p(\pi)$ into three parts $ p^{m}_{k}(\pi):= \{\{i,j\}\in p(\pi): i<k<j\} $, $ p^l_k(\pi):= \{\{i,j\}\in p(\pi): i<j<k\} $, and $ p^{r}_{k}(\pi):= \{\{i,j\}\in p(\pi): k<i<j\} $. Denote by $ \pi^m\in \mathcal{P}_{1,2}(n) $ the partition such that $ p(\pi^m) = p_k^m(\pi) $. Let $ s^k_l(\pi^m) = \{ \{t\}\in s(\pi^m):  t<k \} $, and $s^k_r(\pi^m) = \{ \{t\}\in s(\pi^m):  t>k \} $. Denote by $ \pi^l\in \mathcal{P}_{1,2}( s^k_l(\pi^m) )$ the partition such that $ p(\pi^l) = p^l_k(\pi) $. Let $ c_{s^k_l(\pi^m)}: [|s^k_l(\pi^m)|]\to s^k_l(\pi^m) $ be the counting map. Set $ \sigma^{l} = c^{-1}_{s^k_l(\pi^m)}(\pi^l) \in \mathcal{P}_{1,2}( |s^k_l(\pi^m)| )$. We also define $ \pi^r\in \mathcal{P}_{1,2}( s^k_r(\pi^m) ) $ and $ \sigma^r\in \mathcal{P}_{1,2}( |s^k_r(\pi^m)| ) $ in the same way.
	
	The $4$-tuple $ (\pi^m,k,\sigma^l,\sigma^r) $ constructed above is called the \emph{decomposition of $\pi$ with respect to $k$} and is denoted by $ DCP(\pi,k) $. More precisely, the above procedure gives us a (bijective) map $$ DCP: \{ (\pi,k) \in \mathcal{P}_{1,2}(n)\times [n]: \{k\}\in s(\pi) \} \longrightarrow \text{dcp}(n), $$ where $ \text{dcp}(n):= \{ (\pi^m,k,\sigma^{l},\sigma^{r}): \pi^m\in \mathcal{P}_{1,2}(n),k\in \cap p(\pi^m), \sigma^{l}\in \mathcal{P}_{1,2}(|s^k_l(\pi^m)|), \sigma^{r}\in \mathcal{P}_{1,2}(|s^k_r(\pi^m)|) \} $ with $ \cap p(\pi)$ the intersection of the open intervals $\cap_{\{i,j\}\in p(\pi)}\{ m: i<m< j \} $. In fact, the inverse of $ DCP $ is given by $ (DCP)^{-1}(\pi^m,k,\sigma^l,\sigma^r) = (\pi,k) $ with $ \pi\in \mathcal{P}_{1,2}(n) $ the unique partition such that $ p(\pi) = p( \pi^m )\cup p( c_{s^l_k(\pi^m)}(\sigma^{l}) )\cup p( c_{s^r_k(\pi^m)}(\sigma^{r}) ) $.
\end{defn}

\begin{prop}\label{decomposition of W}
	Let $ \pi \in \mathcal{P}_{1,2}(n) $ and $ \{k\}\in s(\pi) $, and $ DCP(\pi,k)=( \pi^m,k,\sigma^{l},\sigma^{r} ) $ be its decomposition. Then 
	$$ W_\pi^T = W^T_{\sigma^{l}}( \text{id}_{\HH}^{\otimes |s^k_l(\pi^m)|+1}\otimes W^T_{\sigma^{r}} ) W^T_{\pi^m} = ( W^T_{\sigma^{l}}\otimes \text{id}_\HH \otimes W^T_{\sigma^{r}} )W^T_{\pi^m}.$$
\end{prop}
\begin{proof}
	This is in fact simply Proposition \ref{W is independent of order} when choosing the correct admissible order $ \varphi $. Let again $ \pi^l = c_{s^k_l(\pi^m)}(\sigma^{l}) $, $ \pi^r = c_{s^k_r(\pi^m)}(\sigma^{r}) $ as in Definition \ref{decomposition of P12}. Let $ L_l $, $ L_m $ and $ L_r $ be the left standard order for $ \pi^l $, $ \pi^m $ and $ \pi^r $. Consider the order $ \varphi $ on $ p(\pi) $ obtained by concatenating $ L_m $, $ L_r $ and $ L_l $, i.e. if $ V,U \in p(\pi^m)$ (similarly for $ p(\pi^r) $, $ p(\pi^l) $), then $ V<_\varphi U$ if $ V<_{L} U $; if $ V\in p(\pi^m) $ and $ V\in p(\pi^r) $ (similarly if $ V\in p(\pi^r) $ and $ V\in p(\pi^l) $ ), then we set $ V<_\varphi U $. To see that $\varphi$ is admissible, one simply note that pairings in $ \pi^r $ and $ \pi^l $ are not $ < $ comparable since they do not overlap, and also that it is impossible to have $ V\in p(\pi^l)\cup p(\pi^r) $ and $ U\in p(\pi^m) $ such that $ V<U $, because $ k $ is in the middle of $ U $ but outside of $ V $. Now by Proposition \ref{W is independent of order}, computing the indices $ r^{\pi,\varphi}_k(m) $, one obtain
	$$ W^T_\pi = W^T_{\pi,\varphi} = ( W^T_{\sigma^{l}}\otimes \text{id}_\HH \otimes W^T_{\sigma^{r}} )W^T_{\pi^m}.$$
\end{proof}

\section{Computation for the Conjugate System}
\subsection{A formula for $ \partial_i\varPhi(\xi_1\otimes \cdots \otimes \xi_n) $}
We now fix a basis $ (e_1,\cdots,e_d) $ of $\HH$, and let $ (f_1,\cdots,f_d) $ be the dual basis of $ (e_1,\cdots,e_d) $, i.e. $ f_i\in \HH $ and $ \langle f_i,e_j\rangle =\delta_{i,j}$.

\begin{defn}
	For each $n\geq 1$, $ 1\leq k< n$ and $ 1\leq i\leq d $, we define the 'partial gradient' operator $\nabla_i^k: \HH^{\otimes n}\to \FF^{\text{fin}}(\HH)\otimes \FF^{\text{fin}}(\HH) $, with $ \FF^{\text{fin}}(\HH) = \oplus_{\text{alg},n=0}^{\infty} \HH^{\otimes n} $ the algebraic direct sum.
	\begin{gather*}
		\nabla_i^k(\xi_1\otimes \cdots \xi_n)=\langle f_i,\xi_k\rangle (\xi_1\otimes \cdots \otimes\xi_{k-1})\otimes (\xi_{k+1}\otimes \cdots \otimes \xi_{n})\in \FF^{\text{fin}}(\HH) \otimes \FF^{\text{fin}}(\HH) ,\quad\forall 2\leq k\leq n-1,\\
		\nabla_i^1(\xi_1\otimes \cdots \xi_n)=	\langle f_i,\xi_1\rangle \Omega\otimes (\xi_2\otimes \cdots \otimes \xi_n),\\
		\nabla_i^n(\xi_1\otimes \cdots \xi_n)=	\langle f_i,\xi_n\rangle (\xi_1\otimes \cdots \otimes \xi_{n-1})\otimes \Omega.
	\end{gather*}
	We also define the partial form of the free difference quotient of variables $ x_1,\cdots,x_n $: for a monomial $ x_{i_1}\cdots x_{i_m} $ with $ m\geq k $, we define
	$$ \partial_i^kx_{i_1}\cdots x_{i_m} := \delta_{i, i_k}x_{i_1}\cdots x_{i_{k-1}}\otimes x_{i_{k+1}}\cdots x_{i_m}. $$
	In particular, when the variables are $ X_T(e_1),\cdots,X_T(e_d) $, we have the relation $\partial_i^k\circ X = (X\otimes X)\nabla_i^k$, or equivalently
	\begin{equation}\label{partial i and nabla}
		\partial_i^k\circ X = (\varPhi^{-1}X\otimes \varPhi^{-1}X)\nabla_i^k,\quad \forall k.
	\end{equation}
	when $ \partial_i^k $ is considered as a map from polynomials to $ \FF_T(\HH)\otimes \FF_T(\HH) $.
\end{defn}

In the following proposition, we will also abuse the notation and consider $ \partial_i $ as a map from polynomials in $ X_T(e_1),\cdots,X_T(e_d) $ to $\FF_T(\HH)\otimes \FF_T(\HH)$. (Since we can always identify $ \mathcal{L}_T(H) $ with its image in $L^2(  \mathcal{L}_T(H),\varphi_\Omega)= \FF_T(\HH) $, this should cause no ambiguity.)

\begin{prop}\label{formula for partial i}
	For each $1\leq i\leq d$,
	\begin{equation}
		\partial_i\varPhi(\xi_1\otimes \cdots \otimes \xi_n) = \sum_{\pi\in \mathcal{P}_{1,2}(n)}\sum_{k\in \cap p(\pi)}(-1)^{|p(\pi)|}\nabla^{\ps(k,\pi)}_iW^T_\pi( \xi_1\otimes \cdots \otimes \xi_n)\in \FF_T(\HH)\otimes \FF_T(\HH),
	\end{equation}
	where $ \bigcap p(\pi) := \bigcap_{\{i,j\}\in \pi}]i,j[ $ is the intersection of the open intervals $]i,j[ = \{k: i<k<j\}$, and $\ps(k,\pi)$ is the position of $k$ in the singletons $s(\pi)$, i.e. if $ s(\pi)=\{\{s_1\},\cdots,\{s_m\}\} $ with $ s_1<\cdots<s_m $, then $ s_{\ps(k,\pi)} = k $.
\end{prop}

\begin{rmk}
	We note that this formula can actually be considered as a generalization of Proposition 5.1 in \cite{MS23}.
\end{rmk}

\begin{proof}
	Using the $T$-Wick product (Theorem \ref{Wick product}), and the 'partial' free difference quotients $\partial_i^k$,
	\begin{align*}
		\partial_i\varPhi(\xi_1\otimes \cdots \otimes \xi_n) =& \partial_i X\bigg(\sum_{\pi\in \mathcal{P}_{1,2}(n)} (-1)^{|p(\pi)|}W_\pi^T(\xi_1\otimes\cdots\otimes \xi_n)\bigg)\\
		=& \sum_{\pi\in \mathcal{P}_{1,2}(n)} (-1)^{|p(\pi)|} \sum_{t=1}^{|s(\pi)|}\partial_i^t XW_\pi^T(\xi_1\otimes\cdots\otimes \xi_n)\\
		\text{\eqref{partial i and nabla}}=& \sum_{\pi\in \mathcal{P}_{1,2}(n)} (-1)^{|p(\pi)|} \sum_{t=1}^{|s(\pi)|}(\varPhi^{-1}X\otimes \varPhi^{-1}X)\nabla_i^t W_\pi^T(\xi_1\otimes\cdots\otimes \xi_n)\\
		=& \sum_{\pi\in \mathcal{P}_{1,2}(n)} (-1)^{|p(\pi)|} (\varPhi^{-1}X\otimes \varPhi^{-1}X)\sum_{k\in s(\pi)}\nabla_i^{\ps(k,\pi)} W_\pi^T(\xi_1\otimes\cdots\otimes \xi_n) 
	\end{align*}
	Now, for each $ k\in s(\pi) $, let $ DCP(\pi,k) = (\pi^m,k,\sigma^{l},\sigma^{r}) $. Apply Proposition \ref{decomposition of W},
	\begin{align*}
		\nabla_i^{\ps(k,\pi)} W_\pi^T(\xi_1\otimes\cdots\otimes \xi_n) &= \nabla_i^{\ps(k,\pi)} ( W^T_{\sigma^{l}}\otimes \text{id}_\HH \otimes W^T_{\sigma^{r}} )W^T_{\pi^m}\\
		&= ( W^T_{\sigma^{l}}\otimes W^T_{\sigma^{r}} ) \nabla_i^{\ps(k,\pi)-|s^k_l(\pi^m)|}W^T_{\pi^m}\\
		&= ( W^T_{\sigma^{l}}\otimes W^T_{\sigma^{r}} ) \nabla_i^{\ps(k,\pi^m)}W^T_{\pi^m},
	\end{align*}
	where in the second to last line we used the fact that $ \ps(k,\pi) = |s^k_l(\pi^m)|+1 $ and that $ \text{id}_\HH $ is exactly the $  (|s^k_l(\pi^m)|+1) $-th factor in $ W^T_{\sigma^{l}}\otimes \text{id}_\HH \otimes W^T_{\sigma^{r}} $. Plug this into the expression for $ \partial_i \varPhi(\xi_1\otimes\cdots \otimes \xi_n) $, we obtain
	\begin{align*}
		&\partial_i\varPhi(\xi_1\otimes \cdots \otimes \xi_n)\\ =&(\varPhi^{-1}X\otimes \varPhi^{-1}X) \sum_{\pi\in \mathcal{P}_{1,2}(n)} (-1)^{|p(\pi)|} \sum_{k\in s(\pi)}  (W^T_{\sigma^{l}}\otimes W^T_{\sigma^{r}} ) \nabla_i^{\ps(k,\pi^m)}W^T_{\pi^m}\\
		=&  (\varPhi^{-1}\otimes \varPhi^{-1})\sum_{\pi\in \mathcal{P}_{1,2}(n)} \sum_{k\in s(\pi)} [ (-1)^{|p(\sigma^{l})|} XW^T_{\sigma^{l}}\otimes (-1)^{|p(\sigma^{r})|}XW^T_{\sigma^{r}} ] (-1)^{|p(\pi^m)|}\nabla_i^{\ps(k,\pi^m)}W^T_{\pi^m}\\
		\intertext{Since DCP is bijective}
		=&(\varPhi^{-1}\otimes \varPhi^{-1})\sum_{ \pi^m\in \mathcal{P}_{1,2}(n) }\sum_{k\in \cap p(\pi^m)}\sum_{\substack{\sigma^{l} \in \mathcal{P}_{1,2}(|s^k_l(\pi^m)|)\\ \sigma^{r} \in \mathcal{P}_{1,2}(|s^k_r(\pi^m)|)  }} [ (-1)^{|p(\sigma^{l})|} XW^T_{\sigma^{l}}\otimes (-1)^{|p(\sigma^{r})|}XW^T_{\sigma^{r}} ] (-1)^{|p(\pi^m)|}\nabla_i^{\ps(k,\pi^m)}W^T_{\pi^m}\\
		\intertext{By Theorem \ref{Wick product}}
		=&(\varPhi^{-1}\otimes \varPhi^{-1})\sum_{ \pi^m\in \mathcal{P}_{1,2}(n) }\sum_{k\in \cap p(\pi^m)} ( \varPhi\otimes \varPhi ) (-1)^{|p(\pi^m)|}\nabla_i^{\ps(k,\pi^m)}W^T_{\pi^m}\\
		=&  \sum_{ \pi^m\in \mathcal{P}_{1,2}(n) }\sum_{k\in \cap p(\pi^m)} (-1)^{|p(\pi^m)|}\nabla_i^{\ps(k,\pi^m)}W^T_{\pi^m}
	\end{align*}
\end{proof}

\subsection{A concrete formula for the conjugate system}
\begin{thm}\label{main}
	Let $ \HH$ be a finite dimensional complex Hilbert space with a standard subspace $ H\subset \HH $, and $T\in \mathcal{T}_{>}(H)$ be a crossing symmetric and braided twist on $ \HH $ with $ \|T\|=q<1 $. Let $ (e_1,\cdots,e_d) $ be a linear basis of $ \HH $. Consider the twisted Araki-Woods von Neumann algebra $ \mathcal{L}_T(H) $ with the vacuum state and generators $ X_T(e_1),\cdots,X_T(e_d) $. Then the conjugate system $ (\Xi_1,\cdots,\Xi_d) $ for $ ( X_T(e_1),\cdots,X_T(e_d) ) $ with respect to the free difference quotient exists and
	$$ \Xi_i = \sum_{n=0}^{\infty}(-1)^{n}P_{T,2n+1}^{-1}\sum_{ \pi\in B(2n+1) }(W_\pi^T)^*f_i ,\quad \forall 1\leq i\leq d, $$
	where $ (f_1,\cdots,f_d) $ is the dual basis of $ ( e_1,\cdots,e_d ) $ in $ \HH $, and $ B(2n+1)\subseteq \mathcal{P}_{1,2}(2n+1) $ is the set of incomplete matchings such that $ n+1 $ is a singleton and each $ 1\leq k\leq n $ must be paired with an element in $ \{n+2,\cdots,2n+1\} $. Here the adjoint $ (W^T_\pi)^* $ are taken on $ \HH^{\otimes t} $ with respect to the untwisted norm $  \langle \cdot,\cdot\rangle= \langle \cdot,\cdot\rangle_0 $.
\end{thm}
\begin{proof}
	First, we show that the sum $ \sum_{n=0}^{\infty}(-1)^{n}P_{T,2n+1}^{-1}\sum_{ \pi\in B(2n+1) }(W_\pi^T)^*f_i $ indeed converges in $ L^2(\mathcal{L}_T(H)) = \FF_T(\HH) $. By Theorem 6 \cite{Bo98}, we have $ \| P_{T,2n+1}^{-1} \|\leq \omega(q)^{-2n-1} $ where $ \omega(q)^2 = (1-q^2)^{-1}\prod_{k=1}^{\infty}(1-q^k)(1+q^k)^{-1} $. Also, for any $ \pi\in B(2n+1) $, since each $ 1\leq k\leq n $ is paired with an element in $ \{n+2,\cdots,2n+1\} $, we have $ \Cr(\pi)\geq \frac{n(n+1)}{2} $, hence $ \|W^T_\pi\|\leq \|C_1\|^n q^{\frac{n(n+1)}{2}}\leq d^{n/2}\|S_H\|^n q^{\frac{n(n+1)}{2}}$. Therefore, we have
	\begin{align*}
		&\| \sum_{n=0}^{\infty}(-1)^{n}P_{T,2n+1}^{-1}\sum_{ \pi\in B(2n+1) }(W_\pi^T)^*f_i \|_{\FF_T(\HH)}= \| \sum_{n=0}^{\infty}(-1)^{n}P_{T,2n+1}^{-{1/2}}\sum_{ \pi\in B(2n+1) }(W_\pi^T)^*f_i \|_{\FF_0(\HH)}\\ &\leq \sum_{n=1}^{\infty}(2n+1)!\omega(q)^{-n-1/2}d^{n/2}\|S_H\|^n q^{\frac{n(n+1)}{2}}\|f_i\| <\infty
	\end{align*}
	which converges since $ q^{\frac{n(n+1)}{2}}\lesssim e^{-cn^2} $ while $ (2n+1)!\lesssim e^{c'n\ln n} $.
	
	Now, by linearity, we only need to show that for all $ \xi_1,\cdots,\xi_k \in \HH$,
	\begin{equation}\label{check Xi is conjugate system}
		\langle \Xi_i, \xi_1\otimes\cdots\otimes \xi_k\rangle_T =  \langle \Omega\otimes \Omega, \partial_i \varPhi( \xi_1\otimes\cdots\otimes \xi_k ) \rangle_{\FF_T(\HH)\otimes \FF_T(\HH)}.
	\end{equation}
	Note that for all $ \pi\in B(2n+1) $, $ \pi $ only has one singeleton $ n+1 $, hence $ (W_\pi^T)^*f_i \in \HH^{\otimes 2n+1} $, and in particular, $\langle \Xi_i, \xi_1\otimes\cdots\otimes \xi_k\rangle_T = 0  $ if $k$ is even. Similarly, we also have $ \langle \Omega\otimes \Omega, \partial_i \varPhi( \xi_1\otimes\cdots\otimes \xi_k ) \rangle_{\FF_T(\HH)\otimes \FF_T(\HH)} = 0$ if $k$ is even. So, we may assume that $ k=2m+1 $ is odd. Now, the left-hand side of equation \eqref{check Xi is conjugate system} is
	\begin{align*}
		\langle \Xi_i, \xi_1\otimes\cdots\otimes \xi_k\rangle_T=&\langle \sum_{n=0}^{\infty}(-1)^{n}P_{T,2n+1}^{-1}\sum_{ \pi\in B(2n+1) }(W_\pi^T)^*f_i,\xi_1\otimes \cdots\otimes \xi_{2m+1} \rangle_T\\
		=& \langle (-1)^m P_{T,2m+1}^{-1}\sum_{ \pi\in B(2m+1) }(W_\pi^T)^*f_i,\xi_1\otimes \cdots\otimes \xi_{2m+1} \rangle_T\\
		=&\langle \sum_{ \pi\in B(2m+1) }(W_\pi^T)^*f_i,(-1)^m\xi_1\otimes \cdots\otimes \xi_{2m+1} \rangle\\
		=&\langle  f_i,\sum_{ \pi\in B(2n+1) }(-1)^m W_\pi^T( \xi_1\otimes \cdots\otimes \xi_{2m+1} ) \rangle.
	\end{align*}
	The right-hand side of equation \eqref{check Xi is conjugate system} is
	\begin{align*}
		 &\langle \Omega\otimes \Omega, \partial_i \varPhi( \xi_1\otimes\cdots\otimes \xi_{2m+1} ) \rangle_{\FF_T(\HH)\otimes \FF_T(\HH)}\\
		 \text{(Prop. \ref{formula for partial i})}=&\langle \Omega\otimes \Omega, \sum_{\pi\in \mathcal{P}_{1,2}(2m+1)}\sum_{l\in \cap p(\pi)}(-1)^{|p(\pi)|}\nabla^{\ps(l,\pi)}_i W^T_\pi( \xi_1\otimes\cdots\otimes \xi_{2m+1} ) \rangle.
	\end{align*}
	Note that, $ \langle \Omega\otimes \Omega,\nabla^{\ps(l,\pi)}_i W^T_\pi( \xi_1\otimes\cdots\otimes \xi_{2m+1} ) \rangle $ is nonzero only when $ |s(\pi)|=1 $. But since $ l\in \cap p(\pi) $, this forces that $ l=n+1 $ and each $1\leq t\leq n $ is paired with an element in $ \{n+2,\cdots,2n+1\} $, i.e. $ \pi\in B(2n+1) $. Thus, we have
	\begin{align*}
		&\langle \Omega\otimes \Omega, \partial_i \varPhi( \xi_1\otimes\cdots\otimes \xi_{2m+1}) \rangle_{\FF_T(\HH)\otimes \FF_T(\HH)}=\langle \Omega\otimes \Omega, \sum_{\pi\in B(2m+1)}(-1)^m \nabla^{1}_iW^T_\pi( \xi_1\otimes\cdots\otimes \xi_{2m+1} )\rangle\\
		=&\langle (\nabla^{1}_i)^*(\Omega\otimes \Omega), \sum_{\pi\in B(2m+1)}(-1)^m W^T_\pi( \xi_1\otimes\cdots\otimes \xi_{2m+1} )\rangle\\
		=&\langle f_i, \sum_{\pi\in B(2m+1)}(-1)^m W^T_\pi( \xi_1\otimes\cdots\otimes \xi_{2m+1} )\rangle = \langle \Xi_i, \xi_1\otimes\cdots\otimes \xi_{2m+1}\rangle_T.
	\end{align*}
\end{proof}

While we do not know the formula for $P_{T,n}^{-1} $, it is relatively easy to compute $ W^*_T $ for $ \pi\in B(2m+1) $ using $ C_1^*\Omega = \sum_{i=1}^{d}S_H^*e_i\otimes e_i $ with $ \{e_i\}_{i} $ an orthonormal basis of $\HH$. For example, if $ \pi=\{\{1,5\},\{2,4\},\{3\}\}\in B(5) $, then $ W_{\pi} = C_1T_2C_1T_2T_3T_4 $ and $ W^*_{\pi}f = \sum_{i,j=1}^{d}T_4T_3T_2T_4(S_H^*e_i\otimes e_i\otimes S_H^*e_j\otimes e_j\otimes f) $. In general, $ W^*_\pi $ is first creating $|p(\pi)|$ copies of $\sum_{i} S_H^*e_i\otimes e_i$ on the left and then applying $ T_i $'s in the 'reverse' order.

We also note that for the $q$-Gaussian algebras, our formula in fact coincides with the formula given by \cite{MS23} using the free right annihilation operators $ a_r(e_i) $'s. This can be shown by the fact that the adjoint of $ a_r(e_i) $ taken in the $q$ norm acts on $ \HH^k $ as $ P_{qF,k+1}^{-1}a_r^*(e_i)P_{qF,k} $ (after expanding $ P_{qF,k} $ in terms of $ T_i $'s).

\subsection{Factoriality and type of $ \mathcal{L}_T(H) $}
Now, since $ (X_T(e_1),\cdots,X_T(e_d)) $ has a conjugate system, it has finite free Fisher information. If we choose $ e_1,\cdots, e_d $ to be the eigenvectors of $ \Delta_H$, then by Theorem \ref{modular data on frakA}, $ X_T(e_1),\cdots,X_T(e_d) $ are also eigenoperators for $ \sigma^\Omega $ with the same eigenvalues. Therefore we can now apply the main results of \cite{Ne17} quoted below.

\begin{thm}[\cite{Ne17}, Theorem A and B]
	Let $M$ be a von Neumann algebra with a faithful normal state $\varphi$. Suppose $M$ is generated by a finite set $ G= G^* $, and $ |G|\geq 2 $ of eigenoperators of the modular automorphism group $\sigma^\varphi$ with finite free Fisher information. Then $ (M^\varphi)'\cap M =\CC$. In particular, the centralizer $ M^\varphi$ is a $ \text{II}_1 $ factor and if $ K < \RR^\times_*$ is the closed subgroup generated by the eigenvalues of $G$ then
	$$ M  \text{ is a factor of type } \begin{cases}
		\text{III}_1,\quad \text{if } K = \RR^\times_*\\
		\text{III}_\lambda, \quad \text{if } K = \lambda^\ZZ, 0<\lambda<1\\
		\text{II}_1,\quad \text{if } K = \{1\}.
	\end{cases}$$
	Moreover, $ M^\varphi $ does not have property $\Gamma$, and if $ M $ is a type $ \text{III}_\lambda $ factor, $0<\lambda<1$, then $ M $ is full.
\end{thm}

\begin{cor}\label{factoriality and type}
	Let $ \HH$ be a complex Hilbert space with dimension $ 2\leq \text{dim}\,\HH<\infty $, $ H\subset \HH $ be a standard subspace, and $T\in \mathcal{T}_{>}(H)$ be a crossing symmetric and braided twisted on $ \HH $ with $ \|T\|=q<1 $. Consider the $T$-twisted Araki-Woods algebra $\mathcal{L}_T(H)$. Then $ (\mathcal{L}_T(H)^{\varphi_\Omega})'\cap \mathcal{L}_T(H)=\CC $ and $ \mathcal{L}_T(H)^{\varphi_\Omega} $ is a $\text{II}_1$ factor that does not have property $ \Gamma $. Moreover, if $ G < \RR^\times_*$ is the closed subgroup generated by the eigenvalues of the modular operator $\Delta_H$, then $ \mathcal{L}_T(H) $ is a factor of type
	$$\begin{cases}
		\text{III}_1,\quad \text{if } G = \RR^\times_*\\
		\text{III}_\lambda, \quad \text{if } G = \lambda^\ZZ, 0<\lambda<1\\
		\text{II}_1,\quad \text{if } G = \{1\}.
	\end{cases}$$
	If $ G = \lambda^\ZZ, 0<\lambda<1 $, then $ \mathcal{L}_T(H) $ is full. \qed
\end{cor}

\section{Isomorphism via Free Monotone Transport}
\subsection{Quasi-free difference quotient $ \eth_i $}
We recall some basic notations about nontracial free monotone transport following \cite{nelson2015free} (also \cite{Ne17}).
\begin{defn}\label{quasifree assump}
	If $ X_1,\cdots,X_d \in M$ are self-adjoint and there is an $n\times n$ matrix $A>0$ which determines the modular operator of $ \varphi $:
	\begin{align*}
		\varphi(X_k X_j) &= \left[ \frac{2}{1+A} \right]_{jk} = \left[ \frac{2A}{1+A} \right]_{kj}  \\
		\sigma^\varphi_{-i}(X_j) &= \sum_{k=1}^{d}[A]_{jk}X_k.
	\end{align*}
	Then the quasi-free difference quotients are defined on $ \CC\langle X_1,\cdots,X_d \rangle $ as
	$$ \eth_i := \sum_{k=1}^{d}\left[ \frac{2}{1+A} \right]_{kj}\partial_k. $$
\end{defn}

In particular, when $ X_i = X_T(e_i) \in \mathcal{L}_T(H) $ with $ e_1,\cdots,e_d\in H $ a orthonomal basis of the real Hilbert space $ (H,\text{Re}\langle \cdot,\cdot\rangle_{\HH}) $, then $ A $ is the matrix for $ \Delta_H^{-1} $ in the basis $ (e_1,\cdots,e_d) $ and
$$ \varphi_\Omega(X_T(e_k) X_T(e_j))  =\langle e_k,e_j\rangle_T = \left[ \frac{2A}{1+A} \right]_{kj}  =\left[ \frac{2}{1+\Delta_H} \right]_{kj} .$$ (Note that $ A^T = A^{-1}=\bar{A} $.) We refer to \cite{shlyakhtenko1997free} for the detailed relation between the standard subspace $H\subset \HH$ and the covariance matrix $\frac{2A}{1+A}$ of a given basis of $H$.

The conjugate system $ (\Theta_1,\cdots,\Theta_d) $ of $ (X_1,\cdots,X_d) $ with respect to $ \eth_i $'s is again defined as $ \eth_i^*(1\otimes 1) \in L^2(M,\varphi)$ which clearly satisfies the relation
$$ \Theta_j =\sum_{k=1}^{d}\overline{\left[ \frac{2}{1+A} \right]_{kj}} \Xi_k=\sum_{k=1}^{d}\left[ \frac{2}{1+A} \right]_{jk} \Xi_k. $$

We now turn to define the norm on power series following \cite{nelson2015free}. Assume $ X_1,\cdots,X_d $ are as in Definition \ref{quasifree assump}. Let $ \mathscr{P} = \CC\langle X_1,\cdots,X_d \rangle $ be the $*$-algebra generated by $ X_1,\cdots,X_d $ in $M$. Denote $ [d]^* $ the set of all finite words in $[d]$. And for $ w = i_{1}i_2\cdots i_k \in [d]^*$, we denote the length of $w$ by $ |w| := k $ and the monomial $X_w := X_{i_1}\cdots X_{i_k}$. So a general polynomial $P\in \mathscr{P}$ can be written as $ P = \sum_{w\in [d]^*} c(w)X_w $. We now consider the norm $ \|\cdot\|_R $ on $\mathscr{P}$,
$$ \| P \|_R = \sum_{w\in [d]^*}|c(w)|R^{|w|}, $$
and denote $ \mathscr{P}^{(R)} $ the completion of $ (\mathscr{P}, \|\cdot\|_R) $.

Define the $ \sigma^\varphi $-cyclic rearrangement $\rho: \mathscr{P} \to \mathscr{P} $,
$$ \rho( X_{j_1}\cdots X_{j_m} ) = \sigma^\varphi_{-i}(X_{j_m})X_{j_1}\cdots X_{j_{m-1}},$$ and define also
$$ \|P\|_{R,\sigma}= \|P\|_{R,\sigma^\varphi} := \sum_{n=0}^{\text{deg}P}\sup_{k\in \ZZ}\| \rho^k( \pi_n(P) ) \|_R \in [0,\infty],$$
where $ \pi_n(P) = \sum_{w\in [d]^*,|w|=n} c(w)X_w $ is the degree $n$ part of $ P $. The completion of $\{ P\in \mathscr{P}: \|P\|_{R,\sigma} < \infty\}$ with respect to $ \|\cdot\|_{R,\sigma} $ is denoted $ \mathscr{P}^{(R,\sigma)} $. And the set of $\sigma$-cyclically symmetric elements is denoted
$$ \mathscr{P}^{(R,\sigma)}_{c.s.}: = \{ P\in  \mathscr{P}^{(R,\sigma)}: \rho(P) = P \}. $$
In particular, for $ P\in \mathscr{P}^{(R,\sigma)}_{c.s.} $, $ \|P\|_{R,\sigma} = \|P\|_{R} $.

Consider the potential $\displaystyle V = \frac{1}{2}\sum_{j,k=1}^{d}\left[ \frac{1+A}{2} \right]_{jk}X_kX_j +W \in  \mathscr{P}^{(R)}$, we say that the state $\varphi$ satisfies the free Gibbs law with potential $V$, if it solves the Schwinger-Dyson equation
$$ \varphi( \mathcal{D}_i(V) P )= \varphi\otimes \varphi^{op}(\eth_i P),\forall 1\leq i\leq d,P\in \mathscr{P}, $$
where $ \mathcal{D}_i $ is the $\sigma$-cyclic derivative
$$ \mathcal{D}_i(X_{k_1}\cdots X_{k_n}):= \sum_{l=1}^{n}\left[\frac{2}{1+A}\right]_{jk_l}\sigma_{-i}(X_{k_{l+1}}\cdots X_{k_{n}})X_{k_1}\cdots X_{k_{l-1}}. $$

If we choose a standard subspace $ H\subseteq \HH $ with covariance $ \langle e_k,e_j\rangle = \left[\frac{2}{1+A}\right]_{jk} $ and consider the free Araki-Woods algebra $ \mathcal{L}_0(H) $ with generators $ X_0(e_i) $'s, then the vacuum state $ \varphi_\Omega $ satisfies the free Gibbs law with the potential $\displaystyle V_0 := \frac{1}{2}\sum_{j,k=1}^{d}\left[ \frac{1+A}{2} \right]_{jk}X_kX_j \in \mathscr{P}^{(R,\sigma)}_{c.s.},\forall R>0$.

Suppose now that $(M,\varphi)$ with generators $ (X_1,\cdots,X_d) $ has the free Gibbs law with potential $V$. \cite{nelson2015free} shows that when $ V$ is close to $ V_0 $ in the sense that $ W=W^* = V-V_0 \in \mathscr{P}^{(R,\sigma)}_{c.s.} $ for some $ R> 4\|A\|+1 $ and $ \|W\|_{R} = \|W\|_{R,\sigma}<C_R $ with $ C_R $ certain small constant, then there is a free monotone transport $ (H_1,\cdots,H_d)  $ with $ H_i\in \mathscr{P}^{(R-1,\sigma)}_{c.s.} $ from $ (X_0(e_1),\cdots,X_0(e_d)) $ to $ (X_1,\cdots,X_d) $ (i.e. $ (H_1(X_0(e_1),\cdots,X_0(e_d)), \cdots,H_d(X_0(e_1),\cdots,X_0(e_d)) $ and $ (X_1,\cdots,X_d) $ have the same joint distribution). In particular, $ M = W^*(X_1,\cdots,X_d) $ is isomorphic to the free Araki-Woods algebra
$$ M = W^*(X_1,\cdots,X_d) \simeq W^*( X_0(e_1),\cdots,X_0(e_d) )= \mathcal{L}_0(H). $$

One way to guarantee that the potential $ V $ is indeed in $ \mathscr{P}^{(R,\sigma)}_{c.s.} $ is through the conjugate variables with respect to $ \eth_i $'s. By the proof of Theorem 4.5 \cite{nelson2015free}, if the conjugate system $ (\Theta_1,\cdots,\Theta_d) $ exists in $ \mathscr{P}^{(R)} $, then the free Gibbs potential of $ \varphi $ is
$$ V = \mathcal{N}^{-1}\left( \sum_{j,k=1}^{d}\left[ \frac{2}{1+A} \right]_{jk}\Theta_k X_j \right), $$
where $ \mathcal{N} $ is the number operator $ \mathcal{N}(X_w) = |w|X_w $ for all $w\in [d]^*$. Moreover, we automatically have $ V\in \mathscr{P}^{(R,\sigma)}_{c.s.} $. In this case, since $$ \|W\|_{R} = \|V-V_0\|_{R}\leq \sum_{j,k=1}^{d}\bigg|\left[ \frac{2}{1+A} \right]_{jk}\bigg|\|\Theta_k-X_k\|_{R}R,$$ to invoke the free monotone transport (Corollary 3.18 \cite{nelson2015free}), it suffices to show that the difference $ \|\Theta_k-X_k\|_{R} $ is small enough for all $1\leq k\leq d$.

\subsection{Power series expansion of the conjugate system for $ \mathcal{L}_T(H)$}
We now fixed the $W^*$-probability space $ (\mathcal{L}_T(H),\varphi_\Omega) $ with  generator $ (X_1,\cdots,X_d) = (X_T(e_1),\cdots,X_T(e_d)) $ where $ e_1,\cdots,e_d\in H $ is a orthonomal basis of the real Hilbert space $ (H,\text{Re}\langle \cdot,\cdot\rangle_{\HH}) $. In particular, we have $ A = \Delta_H^{-1} $ and $\displaystyle \left[ \frac{2}{1+A} \right]_{jk} = \langle e_k,e_j\rangle  $. Therefore, by Theorem \ref{main}, the quasi-free conjugate system is
\begin{align*}
	\Theta_j &= \sum_{k=1}^{d}\left[ \frac{2}{1+A} \right]_{jk} \Xi_k= \sum_{k=1}^{d}\langle e_k,e_j \rangle \Xi_k\\
	&=\sum_{n=0}^{\infty}(-1)^{n}P_{T,2n+1}^{-1}\sum_{ \pi\in B(2n+1) }(W_\pi^T)^* \sum_{k=1}^{d}\langle e_k,e_j \rangle f_k\\
	&= \sum_{n=0}^{\infty}(-1)^{n}P_{T,2n+1}^{-1}\sum_{ \pi\in B(2n+1) }(W_\pi^T)^*e_j,
\end{align*}
where we used the fact that $ \sum_{k=1}^{d}\langle e_k,e_j \rangle f_k = e_j $. (This is because $ \langle  \sum_{k=1}^{d}\langle e_k,e_j \rangle f_k,e_i\rangle = \sum_{k=1}^d \delta_{k,i}\langle e_j,e_k \rangle = \langle e_j,e_i\rangle$ for all $ i $.)

\begin{cor}
	The conjugate system $ (\Theta_1,\cdots,\Theta_d) $ for $ (X_1,\cdots,X_d) $ with respect to the quasi-free difference quotients exists, and
	$$ \Theta_i = \sum_{n=0}^{\infty}(-1)^{n}P_{T,2n+1}^{-1}\sum_{ \pi\in B(2n+1) }(W_\pi^T)^*e_i,\quad \forall 1\leq i\leq d.\qed$$
\end{cor}

Using the formula of $T$-Wick product, we can also formally expand $ \Theta_i $ as a power series in $ X_T(e_1),\cdots,X_T(e_d) $. Just like in \cite{MS23}, the presence of the operators $ P^{-1}_{T,k} $'s in the expression of $ \Theta_i $ prevents us from having a explicit formula of $ \Theta_i $. Nevertheless, the estimate $ \|W^T_\pi\|\leq \|C_1\|^n \|T\|^{\frac{m(m+1)}{2}}= \|C_1\|^n q^{\frac{m(m+1)}{2}} $ for $\pi\in B(2m+1)$ always guarantees the convergence of the summation of the terms for $\Theta_i $ even in the power series norm.

We again use the short notation: for each word $ w = i_1 i_2 \cdots i_k\in [d]^* $ with length $ |w|=k $, we denote $ e_w :=e_{i_1}\otimes\cdots \otimes e_{i_k}\in \HH^{\otimes k} $, $ f_w: = f_{i_1}\otimes \cdots \otimes f_{i_k}$, and $ X_w = X_T(e_{i_1})\cdots X_T(e_{i_1}) $. Note that since $ \delta_{k,j}=\text{Re}\langle e_k,e_j\rangle = \langle e_k,\frac{1+\Delta_H}{2}e_j \rangle $, the dual basis $\{f_i\}_{i\leq d}$ is $ f_i = \frac{1+\Delta_H}{2}e_i $, hence $ \|f_i\|\leq \frac{1+\|\Delta_H\|}{2}\leq \|\Delta_H\|=\|S_H\|^2 $.

Denote $ \Theta_i = \sum_{m=0}^{\infty}\sum_{w\in [d]^*,|w|=2m+1}  \alpha(w,i)e_w $, then we have the estimate as in the proof of Theorem \ref{main}
$$ | \alpha(w,i) | = |\langle f_w, (-1)^{m}P_{T,2m+1}^{-1}\sum_{ \pi\in B(2m+1) }(W_\pi^T)^*e_i\rangle_{\FF_0(\HH)}|\leq (2m+1)!\omega(q)^{-2m-1}d^{m/2}\|S_H\|^{m+2} q^{\frac{m(m+1)}{2}}.$$
Denote also $ W^T_\pi e_w = \sum_{|\nu|=|s(\pi)|}\beta(\pi,w,\nu)e_\nu $, then similarly $ |\beta(\pi,w,\nu  )|\leq \|\Delta_H\|\|W^T_\pi\|\leq \|\Delta_H\|\|C_1\|^{|p(\pi)|}< d^{|\pi|/2}\|S_H\|^{|\pi|+2} $.
Now apply Theorem \ref{Wick product}, the power series expansion of $ \Theta_i $ is: (we are in fact estimating $ \varPhi(\Theta_i) $ since $ \Theta_i \in \FF_T(\HH)$.)
\begin{align*}
	\varPhi(\Theta_i) =& X_i+\sum_{m=1}^{\infty}\sum_{|w|=2m+1}  \alpha(w,i)\sum_{\pi\in \mathcal{P}_{1,2}(2m+1)}(-1)^{|p(\pi)|}XW^T_{\pi}e_w\\
	=&  X_i+\sum_{m=1}^{\infty}\sum_{|w|=2m+1}  \alpha(w,i)\sum_{\pi\in \mathcal{P}_{1,2}(2m+1)}(-1)^{|p(\pi)|}\sum_{|\nu| = |s(\pi)|} \beta(\pi,w,\nu) X_{\nu}
\end{align*}
For any $ R>1 $, we have
\begin{align*}
	\|\varPhi(\Theta_i) - X_i\|_{R}&=\|\sum_{m=1}^{\infty}\sum_{|w|=2m+1}  \alpha(w,i)\sum_{\pi\in \mathcal{P}_{1,2}(2m+1)}(-1)^{|p(\pi)|}\sum_{|\nu| = |s(\pi)|} \beta(\pi,w,\nu) X_{\nu}\|_{R}\\  \leq& \sum_{m=1}^{\infty} d^{2m+1}(2m+1)!\omega(q)^{-2m-1}d^{m/2}\|S_H\|^{m+2} q^{\frac{m(m+1)}{2}}(2m+1)! d^{m+1/2}\|S_H\|^{2m+3}R^{2m+1}\\
	=& \sum_{m=1}^{\infty} d^{7m/2+1/2}((2m+1)!)^2 \omega(q)^{-2m-1} q^{\frac{m(m+1)}{2}}\|S_H\|^{3m+5}R^{2m+1}<\infty.
\end{align*}
For $ q<1/2 $, $ \omega(q)^{-2}=(1-q^2)\prod_{k=1}^{\infty}(1-q^k)^{-1}(1+q^k)\leq \prod_{k=1}^{\infty}(1-2^{-k})^{-1}(1+2^{-k}) =: c  $ is bounded by a universal constant. So for $ q<1/2 $, 
\begin{align*}
	\|\varPhi(\Theta_i) - X_i\|_{R}
	\leq& \sum_{m=1}^{\infty} d^{4m} (3m)^{6m} c^{3m} q^{\frac{m(m+1)}{2}}\|S_H\|^{3m+5}R^{2m+1}.
\end{align*}
which converges to $ 0 $ as $ q $ tends to $0$ for any fixed $R$. In particular, we have
$$ \lim_{\|T\|\to 0} \|\varPhi(\Theta_i) - X_i\|_{R} = 0,$$
where the limit converges uniformly for the twist $ T $ as $ \|T\|=q\to 0 $.

Therefore by Corollary 3.18 \cite{nelson2015free}, we obtain:
\begin{cor}
	For any standard subspace $ H\subset \HH $ with $ 2\leq \text{dim}\,\HH <\infty$, there is a constant $ q_H>0 $ depending on $H$, such that for any compatible crossing symmetric braided twist $ T\in \mathcal{T}_{>}(H) $ with $ \|T\|<q_H $, the $T$-twisted Araki-Woods algebra $ \mathcal{L}_T(H) $ is isomorphic to the free Araki-Woods algebra $ \mathcal{L}_{0}(H) $.\qed
\end{cor}

\appendix
\section{Appendix: Non-injectivity of $ \mathcal{L}_T(H) $ for infinite dimensional $ \HH $}
In \cite{Fum01}, it is shown that if a standard subspace $ H\subset \HH $ satisfies $$ \exists C\geq 1, \frac{\text{dim}E_{\Delta_H}([1,C])\HH}{C} >\frac{16}{(1-|q|)^2}, $$
where $ E_{\Delta_H} $ is the spectral measure of $ \Delta_H $, then the $ q $-Araki-Woods algebra $ \mathcal{L}_{qF}(H) $ is non-injective for any $ -1<q<1 $. (This is later improved by \cite{nou2004non}.) We show that the same arguments in \cite{Fum01} generalize to general twist $T$ easily once we proved the following Lemma \ref{preservation operator} about preservation operators on $\FF_T(\HH)$. In this appendix, the Hilbert space $\HH$ can be either finite or infinite dimensional.  

\begin{lem}\label{preservation operator}
	Let $\HH$ be a complex Hilbert space, $ T\in \mathcal{T}_\geq $ be a twist on $\HH$, $ A\in B(\HH) $ be a bounded operator, and $ \Lambda(A) $ be the preservation operator on the full Fock space $ \FF_0(\HH)=\bigoplus_{n=0}^{\infty}\HH^{\otimes n} $, i.e. $ \Lambda(A): f_1\otimes \cdots \otimes f_n \mapsto (Af_1)\otimes \cdots \otimes f_n $, $ \Lambda(A)\Omega = 0 $.
	\begin{enumerate}[a)]
		\item For each $n$, $ \Lambda(A)R_{T,n} $ maps $ \ker P_{T,n} $ into $ \ker P_{T,n} $, and therefore defines a linear map $ \Lambda_{T}(A): \HH_{T,n}^0\to \HH_{T,n}^0 $ ($\HH_{T,n}^0 = \HH^{\otimes n}/\ker P_{T,n}$),
		$$ \Lambda_{T}(A): [\Phi_n]\mapsto [ \Lambda(A)R_{T,n}\Phi_n ],\quad \forall \Phi_n\in \HH^{\otimes n}.$$
		\item For each $n$, $ \Lambda_{T}(A):\HH_{T,n}^0\to \HH_{T,n}^0 $ is bounded and thus extends to the completion $ \HH_{T,n} $. Moreover, $ \Lambda_{T}(A) $ extends to a unique closed and densely defined operator in $ \FF_T(\HH) $. In fact, $ \Lambda_{T}(A^*) = \Lambda_{T}(A)^*$.
		\item $ \|\Lambda_T(A)|_{ \HH_{T,n}  }\| \leq \|A\|{\sum_{i=0}^{n}\|T\|^i} $. In particular, if $ \|T\|<1 $, then $\displaystyle \|\Lambda_{T}(A)\|\leq \frac{\|A\|}{{1-\|T\|}} $.
	\end{enumerate}
\end{lem}
\begin{proof}
	\begin{enumerate}[a)]
		\item Recall that by the self-adjointness of $ P_{T,n} $, and the relation $ P_{T,n} = (1\otimes P_{T,n-1})R_{T,n} $, we have also $  P_{T,n} = R^*_{T,n}(1\otimes P_{T,n-1}) $. Now, for a $ \Phi_n \in \ker P_{T,n} $, we have
		$$ P_{T,n}\Lambda(A)R_{T,n}\Phi_n = R^*_{T,n}(1\otimes P_{T,n-1})\Lambda(A)R_{T,n}\Phi_n = R^*_{T,n}\Lambda(A)(1\otimes P_{T,n-1})R_{T,n}\Phi_n =  R^*_{T,n}\Lambda(A)P_{T,n}\Phi_n=0. $$
		
		\item By the positivity of $ P_{T,n} $, we have $ P^2_{T,n} = (1\otimes P_{T,n-1})R_{T,n}R^*_{T,n}(1\otimes P_{T,n-1})\leq \| R_{T,n} \|^2(1\otimes P_{T,n-1}^2)\leq (\sum_{i=0}^{n}\|T\|^i)^2(1\otimes P_{T,n-1}^2) $. Therefore, $$ P_{T,n} \leq (\sum_{i=0}^{n}\|T\|^i)(1\otimes P_{T,n-1}).$$
		Now, for $ \Phi_n\in \HH^{\otimes n} $, we have
		\begin{align*}
			\| \Lambda_{T}(A)[\Phi_n] \|^2_{T} &= \langle  \Lambda(A)R_{T,n}\Phi_n, P_{T,n}\Lambda(A)R_{T,n}\Phi_n\rangle_{0} \leq  (\sum_{i=0}^{n}\|T\|^i)\langle  \Lambda(A)R_{T,n}\Phi_n, (1\otimes P_{T,n-1})\Lambda(A)R_{T,n}\Phi_n\rangle\\
			&=(\sum_{i=0}^{n}\|T\|^i)\langle  (1\otimes P_{T,n-1}^{1/2})R_{T,n}\Phi_n, 
			\Lambda(A)^*\Lambda(A)(1\otimes P_{T,n-1}^{1/2})R_{T,n}\Phi_n\rangle\\
			&\leq (\sum_{i=0}^{n}\|T\|^i)\|\Lambda(A)\|^2\langle  (1\otimes P_{T,n-1}^{1/2})R_{T,n}\Phi_n, 
			(1\otimes P_{T,n-1}^{1/2})R_{T,n}\Phi_n\rangle\\
			&=(\sum_{i=0}^{n}\|T\|^i)\|A\|^2\langle  \Phi_n, 
			R_{T,n}^*(1\otimes P_{T,n-1})R_{T,n}\Phi_n\rangle.
		\end{align*}
		It remains to estimate the positive operator (on $\HH^{\otimes n}$) $$ R_{T,n}^*(1\otimes P_{T,n-1})R_{T,n}= R_{T,n}^*P_{T,n} = P_{T,n}R_{T,n}\geq 0. $$
		But again, we have
		$$ (R_{T,n}^*(1\otimes P_{T,n-1})R_{T,n})^2 = P_{T,n}R_{T,n}R_{T,n}^*P_{T,n}\leq  \|R_{T,n}\|^2P_{T,n}^2,$$
		hence $ R_{T,n}^*(1\otimes P_{T,n-1})R_{T,n}\leq (\sum_{i=0}^{n}\|T\|^i) P_{T,n} $. Plug this into the previous inequality for $\| \Lambda_{T}(A)[\Phi_n] \|^2_{T} $, we obtain
		$$ \| \Lambda_{T}(A)[\Phi_n] \|^2_{T}\leq  (\sum_{i=0}^{n}\|T\|^i)^2\|A\|^2\langle  \Phi_n, 
		P_{T,n}\Phi_n\rangle = (\sum_{i=0}^{n}\|T\|^i)^2\|A\|^2\|[\Phi_n]\|_T^2, $$
		and therefore $ \|\Lambda_{T}(A)\big|_{\HH_{T,n}}\| \leq (\sum_{i=0}^{n}\|T\|^i)\|A\| $.
		
		Finally, it remains to show that $ \Lambda_{T}(A^*) $ is the adjoint of $ \Lambda_{T}(A) $ (and since they are both densely defined, they must also be closable.) But this follows from direct calculation: for all $ [\Phi_n],[\Psi_n]\in \HH_{T,n} $,
		\begin{align*}
			\langle  [\Psi_n],\Lambda_{T}(A)[\Phi_n] \rangle_T &= \langle \Psi_n, P_{T,n}\Lambda(A)R_{T,n}\Phi_n \rangle = \langle  \Psi_n, R^*_{T,n}(1\otimes P_{T,n-1})\Lambda(A)R_{T,n}\Phi_n \rangle\\
			&= \langle  \Psi_n, R^*_{T,n}\Lambda(A)(1\otimes P_{T,n-1})R_{T,n}\Phi_n \rangle =\langle  \Psi_n, R^*_{T,n}\Lambda(A)P_{T,n}\Phi_n \rangle \\&=\langle \Lambda(A^*)R_{T,n}\Psi_n,P_{T,n}R_{T,n} \rangle = \langle \Lambda_{T}(A^*)[\Psi_n], [\Phi_n]\rangle_T,
		\end{align*}
		where again we used the definition formula $ P_{T,n} = R^*_{T,n}(1\otimes P_{T,n-1}) = (1\otimes P_{T,n-1})R_{T,n} $, and the fact that $ \Lambda(A) $ commutes with $ (1\otimes P_{T,n-1}) $.
	\end{enumerate}
\end{proof}
As a direct corollary, we have the following bound for product of annihilation and creation operators similar to Lemma 2.1 \cite{Fum01}.
\begin{lem}\label{Hiai Lem 2.1}
	Assume that $T$ is a twist with $ \|T\|<1 $. If $ e_1,\cdots,e_m $ are orthonormal vectors in $ \HH $, then
	$$ \|\sum_{i=1}^{m} a_T^*(e_i)a_T(e_i)\|\leq \frac{1}{{1-\|T\|}}. $$
\end{lem}
\begin{proof}
	We simply note that $ \sum_{i=1}^{m} a_T^*(e_i)a_T(e_i) = \Lambda_{T}(Q) $ with $ Q\in B(\HH) $ the orthogonal projection onto the subspace $ \text{span}\{e_1,\cdots,e_m\} $.
\end{proof}

The proof of the following theorem is mostly the same as in Theorem 2.3 \cite{Fum01}, we streamline the proof here for reader convenience.
\begin{thm}\label{Hiai noninj}
	Let $ H\subset \HH $ be a standard subspace (of finite or infinite dimension), $T$ be a compatible braided crossing symmetric twist with $ \|T\| <1$. If there is a constant $C\geq 1$ such that
	$$ \frac{\text{dim}\;E_{\Delta_H}([1,C])\HH}{C} >\frac{16}{(1-\|T\|)^2}, $$
	then $ \mathcal{L}_T(H) $ is not injective.
\end{thm}
\begin{proof}
	\begin{enumerate}[a)]
		\item We first assume that $ \Delta_H^{it} $ is almost periodic, i.e. eigenvectors of $ \Delta_H $ generate $\HH$. Let $ f_1,\cdots,f_s $ be an orthonormal family of the eigenspace of $ \Delta_H $ with eigenvalues $1$, and $e_1,\cdots,e_t$ be an orthonormal family of the eigenspace of $ \Delta_H $ with eigenvalues $ \lambda_1,\cdots,\lambda_t\in (1,C] $. By the assumption, we can pick $s,t$ large enough so that $ \frac{s+t}{C}> \frac{16}{{(1-\|T\|)^2} } $. If $\mathcal{L}_T(H)$ is injective, then we must have a homomorphism $ \mathcal{L}_T(H)\otimes_{\text{min}}\mathcal{L}_T(H)'\to C^*(\mathcal{L}_T(H),\mathcal{L}_T(H)')\subseteq B(\FF_T(\HH)) $ sending any $ X = \sum x_i\otimes y_i $ to $ \tilde{X} :=\sum x_iy_i\in C^*(\mathcal{L}_T(H),\mathcal{L}_T(H)') $ for $ x_i\in \mathcal{L}_T(H) $ and $ y_i\in \mathcal{L}_T(H)' $. Therefore, we have inequality $ \|\tilde{X}\| \leq \|X\|_{\text{min}} $. Consider
		\begin{align*}
			X &= \sum_{k=1}^{s} X_T(f_k)\otimes J_\Omega X_T(f_k)J_\Omega + \sum_{k=1}^{t}\left( \frac{1}{1+\lambda_k}X_T(e_k)\otimes J_\Omega X_T(e_k)J_\Omega+ \frac{1}{1+\lambda_k^{-1}}X_T(J_H e_k)\otimes J_\Omega X_T(J_H e_k)J_\Omega \right)\\
			\tilde{X} & = \sum_{k=1}^{s} X_T(f_k)J_\Omega X_T(f_k)J_\Omega + \sum_{k=1}^{t}\left( \frac{1}{1+\lambda_k}X_T(e_k)J_\Omega X_T(e_k)J_\Omega+  \frac{1}{1+\lambda_k^{-1}}X_T(J_H e_k)J_\Omega X_T(J_H e_k)J_\Omega \right)
		\end{align*}
		Then \begin{align*}
			\|\tilde{X}\| &\geq \langle \Omega, \tilde{X}\Omega\rangle = s+ \sum_{k=1}^{t}\left(\frac{1}{1+\lambda_k}\langle S_H e_k,J_He_k\rangle + \frac{1}{1+\lambda_k^{-1}}\langle S_H J_He_k,e_k\rangle\right)\\ &= s+ \sum_{k=1}^{t}\frac{2}{\sqrt{\lambda_k}+\sqrt{\lambda_k^{-1}}} \geq \frac{s+t}{\sqrt{C}}> \frac{4\sqrt{s+t}}{1-\|T\|}.
		\end{align*}
		On the other hand, $ X = X_1+X_2 $ with
		\begin{align*}
			X_1 &:= \sum_{k=1}^{s} a_T^*(f_k)\otimes J_\Omega X_T(f_k)J_\Omega + \sum_{k=1}^{t}\left( \frac{1}{1+\lambda_k}a_T^*(e_k)\otimes J_\Omega X_T(e_k)J_\Omega+ \frac{1}{1+\lambda_k^{-1}}a_T^*(J_H e_k)\otimes J_\Omega X_T(J_H e_k)J_\Omega \right).\\
			X_2 &:= \sum_{k=1}^{s} a_T(f_k)\otimes J_\Omega X_T(f_k)J_\Omega + \sum_{k=1}^{t}\left( \frac{1}{1+\lambda_k}a_T(e_k)\otimes J_\Omega X_T(e_k)J_\Omega+ \frac{1}{1+\lambda_k^{-1}}a_T(J_H e_k)\otimes J_\Omega X_T(J_H e_k)J_\Omega \right).
		\end{align*}
		Therefore by Lemma \ref{Hiai Lem 2.1} and the estimate $ \|X_T(\xi)\|\leq \|a_T(S_H\xi)\|+\|a^*_T(\xi)\|\leq (\|S_H\xi\|+\|\xi\|)/\sqrt{1-\|T\|} $, \begin{align*}
			\|X_1\|_{\text{min}}\leq& \| \sum_{k=1}^{s}a_T^*(f_k)a_T(f_k) + \sum_{k=1}^{t}\left( a_T^*(e_k)a_T(e_k)+a_T^*(J_He_k)a_T(J_H e_k) \right)\|^{1/2}\\ &\times \left( \sum_{k=1}^{s}\| X_T(f_k) \|^2+ \sum_{k=1}^{t}\left( \frac{1}{(1+\lambda_k)^2}\|X_T(e_k)\|^2+ \frac{1}{(1+\lambda_k^{-1})^2}\|X_T(J_H e_k)\|^2\right) \right)^{1/2}\\
			\leq & \frac{1}{1-\|T\|}\left( 4s+ \sum_{k=1}^{t}( \frac{(1+\lambda_k^{1/2})^{2}}{(1+\lambda_k)^2}+ \frac{(1+\lambda_k^{-1/2})^{2}}{(1+\lambda_k^{-1})^2} )\right)^{1/2} \\
			\leq & \frac{2\sqrt{s+t}}{1-\|T\|}.
		\end{align*}
	Similarly, we have $\displaystyle\|X_2\|_{\text{min}}\leq \frac{2\sqrt{s+t}}{1-\|T\|}$, hence $\displaystyle \|X\|_{\text{min}} \leq \frac{4\sqrt{s+t}}{1-\|T\|} $. Contradiction.
	\item If $ \Delta_H^{it} $ is not almost periodic, then the spectrum of $\Delta_H$ must contain an interval $ [a,b] $ with $ 1\leq a<b $. Choose a sequence of orthonormal vectors $ \{e_i\}_{i\geq 1} $ with $ e_i \in E_{\Delta_H}( ( a+\frac{b-a}{i+1}, a+\frac{b-a}{i}])\HH $ and let $$ X^{(N)} = \sum_{i=1}^{N} X_T(e_i)\otimes J_\Omega X_T(e_i)J_\Omega \in \mathcal{L}_T(H)\otimes_{\text{min}}\mathcal{L}_T(H)'.$$
	Then by similar computation as in (a), we have $ \|\widetilde{X^{(N)}}\|\geq \sum_{i=1}^{N}\langle S_H e_K, J_H e_k\rangle \geq Na^{1/2}, $ and
	$$ \|X^{(N)}\|_{\text{min}} \leq \frac{4}{1-\|T\|}N^{1/2}b^{1/2}. $$
	For $ N $ sufficiently large, we have $ \|X^{(N)}\|_{\text{min}} <  \|\widetilde{X^{(N)}}\| $ hence $ \mathcal{L}_T(H) $ is not injective.
	\end{enumerate}
\end{proof}

\begin{cor}
	Let $ T $ be a compatible braided crossing symmetric twist with $ \|T\|<1 $, then $ \mathcal{L}_T(H) $ is non-injective if one of the following conditions holds.
	\begin{enumerate}[a)]
		\item $ \Delta_H = \text{id}_{\HH} $ and $\text{dim}\geq 2$.
		\item There is a constant $C\geq 1$ such that $$ \frac{\text{dim}\;E_{\Delta_H}([1,C])\HH}{C} >\frac{16}{(1-\|T\|)^2}.$$
		\item $ 2\leq \text{dim}\; \HH<\infty $.
	\end{enumerate}
\end{cor}
\begin{proof}
	a) is due to \cite{nou2004non}. b) is Theorem \ref{Hiai noninj}. For c), if $\Delta_H\neq \text{id} $, then by Corollary \ref{factoriality and type}, the centralizer of $ \mathcal{L}_T(H) $ does not have Gamma property, hence $ \mathcal{L}_T(H) $ is non-injective.
\end{proof}

\begin{rmk}
For the $ q $-Araki-Woods algebras (or more generally the $q_{ij}$-twists), those cases already (essentially) cover all  $H\subseteq \HH$ with $ \text{dim} \HH \geq 2 $ and $ -1<q<1 $ (\cite{KSW23}). This is because one can always reduce the general cases to one of the above situations by looking at an expected subalgebra generated by some eigenoperators $ X_{qF}(e_{i_1}),\cdots, X_{qF}(e_{i_s}) $. However, for a general twist $T$, such a proper expected subalgebra may not exist. For example, if $ T=cm^*m $ ($c\neq 0$) is the twist in Example \ref{twist on matrix} for $ \HH = L^2(B(\CC^\infty),\text{Tr}(h\cdot) ) $, then there are no proper $S_H$-invariant subspace $\KK$ of $ \HH = L^2(B(\CC^\infty),\text{Tr}(h\cdot) ) $ such that $ T(\KK\otimes \KK)\subseteq \KK\otimes \KK $ as we already observed in Example \ref{twist on matrix}. Therefore, the injectivity of $ \mathcal{L}_T( L^2(B(\CC^\infty),\text{Tr}(h\cdot) )_{s.a.} ) $ is unclear when we choose $ \min_{i,j}\{h_i/h_j\} $ large enough.
\end{rmk}

\section*{Acknowledgments}

The author would like to thank his Ph.D. advisor Prof. Michael Anshelevich for patiently reading the draft for this paper and helping improve the proofs about the linear orders of contractions.
\nocite{hietarinta1993solving,guionnet2014free,Dab14}
\bibliographystyle{alpha}
\bibliography{twistedGaussian}
\end{document}